\documentclass[12pt]{amsart}
\usepackage{amsfonts}
\usepackage{amsmath}
\usepackage{amssymb}
\usepackage{mathrsfs}
\usepackage{graphicx}
\usepackage{subfigure}
\usepackage{cite}
\usepackage{microtype}
\usepackage{enumerate}
\usepackage{mathrsfs}
\usepackage{tikz}
\usepackage{tikz-cd}
\usepackage{color}
\usepackage{ragged2e}
\allowdisplaybreaks[1]
\usepackage{titletoc}
\usepackage{amsthm}
\allowdisplaybreaks

\usepackage[colorlinks,linkcolor=blue,anchorcolor=red,citecolor=blue]{hyperref}
\usepackage[margin=1in]{geometry}
\usepackage{marginnote}

\usepackage{cleveref}

\usepackage{color}

\newcommand{\la}{\lambda}
\newcommand{\al}{\alpha}
\newcommand{\be}{\beta}

\newcommand{\ve}{\varepsilon}

\newcommand{\R}{\mathbb{R}}

\newcommand{\T}{\mathbb{T}}

\newcommand{\f}{\forall}

\newcommand{\n}[1]{\Vert #1\Vert }
\newcommand{\bn}[1]{\big \Vert #1 \big \Vert }
\newcommand{\bbn}[1]{\Big\Vert #1 \Big \Vert }
\newcommand{\lr}[1]{\left\{ #1 \right\} }
\newcommand{\lrc}[1]{\left[ #1 \right] }
\newcommand{\lrs}[1]{\left( #1 \right) }

\newcommand{\lra}[1]{\langle #1 \rangle }
\newcommand{\blra}[1]{\big\langle #1\big\rangle }
\newcommand{\bblra}[1]{\Big\langle #1\Big\rangle }
\newcommand{\abs}[1]{|#1| }
\newcommand{\babs}[1]{\big | #1 \big|}
\newcommand{\bbabs}[1]{\Big | #1 \Big|}
\newcommand{\wt}[1]{\widetilde{#1}}

\newcommand{\pa}{\partial}
\newcommand{\ol}{\overline}

\begin{document}

\newtheorem{theorem}{Theorem}[section]
\newtheorem{lemma}[theorem]{Lemma}
\theoremstyle{definition}
\newtheorem{definition}[theorem]{Definition}
\newtheorem{example}[theorem]{Example}
\newtheorem{remark}[theorem]{Remark}

\numberwithin{equation}{section}
\newtheorem{proposition}[theorem]{Proposition}
\newtheorem{corollary}[theorem]{Corollary}

\renewcommand{\figurename}{Fig.}

\title[Stability of the inhomogeneous sheared Boltzmann equation]{Global stability of the inhomogeneous sheared Boltzmann equation in torus}

\author[R.-J. Duan]{Renjun Duan}
\address[R.-J. Duan]{Department of Mathematics, The Chinese University of Hong Kong,
Shatin, Hong Kong, P.R.~China}
\email{rjduan@math.cuhk.edu.hk}

\author[S.-Q. Liu]{Shuangqian Liu}
\address[S.-Q. Liu]{School of Mathematics and Statistics, and Key Lab NAA-MOE, Central China Normal University, Wuhan 430079, P.R.~China}
\email{sqliu@ccnu.edu.cn}

\author[S.-L. Shen]{Shunlin Shen}
\address[S.-L. Shen]{School of Mathematical Sciences, University of Science and Technology of China, Hefei 230026,
Anhui Province, P.R.~China}
\email{slshen@ustc.edu.cn}

\begin{abstract}
Homoenergetic solutions to the spatially homogeneous Boltzmann equation have been extensively studied, but their global stability in the inhomogeneous setting remains challenging due to unbounded energy growth under self-similar scaling and the intricate interplay between spatial dependence and nonlinear collision dynamics. In this paper, we introduce an approach for periodic spatial domains to construct global-in-time inhomogeneous solutions in a non-conservative perturbation framework, characterizing the global dynamics of growing energy. The growth of energy is shown to be governed by a long-time limit state that exhibits features not captured in either the homogeneous case or the classical Boltzmann theory.  The core of our proof is the derivation of new energy estimates specific to the Maxwell molecule model. These estimates combine three key ingredients: a low-high frequency decomposition, a spectral analysis of the matrix associated with the second-order moment equation, and a crucial cancellation property in the zero-frequency mode of the nonlinear collision term. This last property bears a close analogy to the null condition in nonlinear wave equations.
\end{abstract}

%\date{\today}


%35Q20  	Boltzmann equations
%35B35  	Stability in context of PDEs
%35B40  	Asymptotic behavior of solutions to PDEs
%76P05  	Rarefied gas flows, Boltzmann equation in fluid mechanics
%76E05  	Parallel shear flows in hydrodynamic stability


\maketitle

\thispagestyle{empty}
\tableofcontents

\section{Introduction}
\subsection{The Homogeneous Boltzmann
Equation for Uniform Shear Flow}
The Boltzmann equation, a fundamental model in kinetic theory, provides a detailed description of rarefied gas dynamics by accounting for particle interactions. Within the framework of the Boltzmann equation,
 the study of uniform shear flow (USF) presents a specialized and intriguing scenario. Here, shearing motion and the induced viscous heating drive the system away from the Maxwellian equilibrium, resulting in a monotonic increase in energy and temperature over time.  Given the relevance of such physical configurations in various applications, understanding the global existence and long-time behavior of the solutions to the Boltzmann equation for USF is of significant interest.

The spatially homogeneous Boltzmann equation for USF is given by:
\begin{equation}\label{equ:USF}
\pa_{t}F- \alpha v_{2}\pa_{v_{1}}F = Q(F, F),
\end{equation}
where $F = F(t, v) \geq 0$ is the velocity distribution function for particles with velocity $v = (v_{1}, v_{2}, v_{3}) \in \mathbb{R}^3$ at time $t \geq 0$, and $\alpha > 0$ is the shear rate.

The Boltzmann collision operator $Q(F_{1},F_{2})$, which describes the effect of binary collisions between particles, admits a standard decomposition into gain and loss terms
\begin{equation*}
Q(F_{1},F_{2})=Q^{+}(F_{1},F_{2})-Q^{-}(F_{1},F_{2})
\end{equation*}
where the gain term is
\begin{align*}
Q^{+}(F_1, F_2)(v) = \int_{\mathbb{R}^3} \int_{\mathbb{S}^2} B(v_{*}-v,\omega)  F_1(v') F_2(v_{*}')  d\omega dv_* ,
\end{align*}
and the loss term is
\begin{align*}
Q^{-}(F_1, F_2)(v) = \int_{\mathbb{R}^3} \int_{\mathbb{S}^2} B(v_{*}-v,\omega)   F_1(v) F_2(v_{*})  d\omega dv_* ,
\end{align*}
with the pre-collision and post-collision velocities related by
\begin{equation}\label{equ:pre,post,velocity}
\left\{
\begin{aligned}
v_*' = &v_* - [(v_* - v) \cdot \omega] \omega, \\
v' =& v + [(v_* - v) \cdot \omega] \omega.
\end{aligned}
\right.
\end{equation}

Notably, the macroscopic moment equations governing the system \eqref{equ:USF} exhibit a distinct structure compared to those derived from the classical Boltzmann equation. To be precise, they are given by
\begin{equation}\label{equ:conservation law}
\left\{
\begin{aligned}
&\frac{d}{dt} \int_{\mathbb{R}^3} F  dv = 0, \\
&\frac{d}{dt} \int_{\mathbb{R}^3} v_{1} F  dv + \alpha \int_{\mathbb{R}^3} v_{2} F  dv = 0, \\
&\frac{d}{dt} \int_{\mathbb{R}^3} v_i F  dv = 0, \quad i = 2, 3, \\
&\frac{d}{dt} \int_{\mathbb{R}^3}  |v|^2 F  dv + 2\alpha \int_{\mathbb{R}^3} v_{1} v_{2} F  dv = 0.
\end{aligned}
\right.
\end{equation}
From these moment equations \eqref{equ:conservation law}, without loss of generality we may assume the conservation laws for mass and momentum
\begin{align*}
\int_{\R^{3}}F(t,v)dv=1,\quad \int_{\R^{3}}v_{i}F(t,v)dv=0,\ i=1,2,3,\ \f t\geq 0.
\end{align*}
However, the energy dynamics present a more complex behavior from the last equation of \eqref{equ:conservation law}. The key coupling term $2\al\int_{\R^{3}}v_{1}v_{2}F(v)dv$ in the energy equation reflects the interplay between shear flow and energy transport.
From a physical perspective, this mathematical structure captures the essential thermodynamics of shear-driven systems. Specifically, while mass and momentum remain strictly conserved, the energy grows monotonically due to irreversible viscous heating. Asymptotic analysis suggests this leads to unbounded energy growth in the long-time limit $t\to \infty$, consistent with the expected behavior of such non-equilibrium systems.

To elucidate these phenomena quantitatively, we focus on the Maxwell molecule model with collision kernel
\begin{align*}
B(v_{*}-v,\omega)=B_{0}(\cos \theta),\quad \cos \theta:=\frac{v_{*}-v}{|v_{*}-v|}\cdot \omega,
\end{align*}
 under the Grad's angular cutoff assumption
\begin{align*}
0\leq  B_{0}(\cos \theta)\lesssim |\cos\theta|.
\end{align*}
The Maxwell molecule model is particularly interesting and illuminating as it admits exact self-similar solutions that preserve the dynamics of energy growth.

\subsection{Self-similar Formulation}
The Maxwellian molecule case possesses a unique mathematical structure that allows the reduction of the Cauchy problem for \eqref{equ:USF} to the study of stationary self-similar solutions via an appropriate temperature scaling transformation. This fundamental property stems from the scaling invariance of the collision operator for Maxwell molecules.

More precisely, the self-similar ansatz takes the form (see for example \cite[Chapter 2]{GS03}
and \cite[Section 5.1]{JNV19})
\begin{equation}\label{equ:self-similar solution}
F(t, v) = e^{-3 \beta t} G \left( \frac{v}{e^{\beta t}} \right)
\end{equation}
where the scaling exponent $\beta$ represents the energy growth rate, and the self-similar profile $G = G(v)$ satisfies the stationary equation
\begin{equation}\label{equ:steady profile,G}
\left\{
\begin{aligned}
&- \beta \nabla_v \cdot (v G) - \alpha v_{2} \pa_{v_{1}} G = Q(G, G),\\
&\int_{\R^{3}}G(v)dv=1,\quad \int_{\R^{3}}vG(v)dv=0,\quad \int_{\R^{3}}|v|^{2}G(v)dv=3.
\end{aligned}
\right.
\end{equation}
Through moment analysis, we derive the exact relationship between the shear strength $\al$ and the energy growth rate $\beta$
\begin{align*}
\int_{\mathbb{R}^3} \left[1, v, |v|^2\right] \left\{ -\beta \nabla_v \cdot (v G) - \alpha v_{2} \pa_{v_{1}} G \right\}  {d}v = 0,
\end{align*}
which yields the identity
\begin{equation}\label{equ:beta,alpha}
\beta = -\alpha \frac{\int_{\mathbb{R}^3} v_{1} v_{2} G dv}{\int_{\mathbb{R}^3} |v|^2 G dv}=
-\frac{\alpha}{3}\int_{\mathbb{R}^3} v_{1} v_{2} G dv.
\end{equation}

Recently, the study of regularity properties and large-velocity asymptotics of \eqref{equ:steady profile,G} has been carried out in \cite{DL21} via a perturbative analysis. To be specific, it shows that
\begin{align}\label{equ:G,profile,1,order}
G(v)=\mu(v)+\al \mu^{\frac{1}{2}}(v)G_{1}(v).
\end{align}
where the correction term satisfies the asymptotic behavior
\begin{align}\label{equ:G,profile,velocity}
\mu^{\frac{1}{2}}(v)G_{1}(v)\sim \frac{2}{b_{0}}v_{1}v_{2}\mu(v),\quad \n{w_{l}\mu^{\frac{1}{2}}G_{1}(v)}_{L_{v}^{\infty}}\leq C_{l}.
\end{align}
Here,
$\mu(v)$ denotes the global Maxwellian distribution
\begin{align*}
\mu(v):=\frac{1}{(2\pi)^{\frac{3}{2}}}e^{-\frac{|v|^{2}}{2}},
\end{align*}
and $w_{l}(v):=\lra{v}^{2l}=(1+|v|^{2})^{l}$ is a polynomial weight function. The constant $b_{0}$
is defined as
\begin{equation}\label{equ:constant,b0}
b_0=3 \pi \int_{-1}^1 B_0(z) z^2\left(1-z^2\right) d z.
\end{equation}
In general, $G(v)$ has to be anisotropic in $v$ due to the shearing
motion,
and any $l$-th order velocity moments of $G(v)$ are finite as long as the
shear rate $\al$ is small enough.  A crucial feature of this self-similar profile, supported by Monte Carlo simulations \cite{GS03}, is its polynomial tail behavior at large velocities. This slow decay property contrasts with the exponential decay of the global Maxwellian distribution.

\subsection{Inhomogeneous Problems and Main Results}
A classical problem in the study of the inhomogeneous Boltzmann equation concerns the stability of the global Maxwellian equilibrium. The first major result in this direction was established by Arkeryd, Esposito, and Pulvirenti in their seminal work \cite{AEP87}. A natural and important subsequent step, motivated by the ubiquitous presence of shear flows in kinetic theory, is to investigate the stability of the self-similar, non-equilibrium profile $G(v)$ that emerges in the uniform shear flow setting.

The inhomogeneous problem presents not only mathematical difficulties but also introduces novel physical phenomena that arise from the intricate interplay between spatial dependence and collisional dynamics.
 To gain insight into these effects, we study the USF Boltzmann equation on a periodic spatial domain
\begin{equation}\label{equ:inhomogeneous,Boltzmann,shear}
\left\{
\begin{aligned}
\pa_{t}F + v\cdot \nabla_{x}F - \alpha v_{2}\pa_{v_{1}}F =& Q(F, F),\\
F(0,x,v)=&F_{0}(x,v), \quad (x,v)\in \T^{3}\times \R^{3}.
\end{aligned}
\right.
\end{equation}

An intriguing phenomenon in this inhomogeneous setting is the long-time behavior of the energy dynamics. The energy growth is characterized by a long-time limit state that exhibits features that are unobservable in either the homogeneous case or the classical Boltzmann equations.

To formulate our result, we introduce the following notation
 \begin{align*}
  P_{0}^{x}\phi=:\int_{\T^{3}}\phi(x)dx,\quad P_{\neq 0}^{x}\phi=:\phi-P_{0}^{x}\phi.
 \end{align*}
We now present our main theorem.

\begin{theorem}\label{thm:main theorem}
Let $G(v)$ be the self-similar profile with $\beta$ given by \eqref{equ:beta,alpha}. For any fixed $l$ large enough, there exist $\al_{0}>0$, $\ve_{0}>0$, $\la_{0}>0$ and $C>0$, such that if $F_{0}(x,v)\geq 0$ and
\begin{align*}
\sum_{k=0,1}\n{w_{l}\pa_{x}^{k}\lrc{F_{0}(x,v)-G(v)}}_{L_{x,v}^{\infty}}\leq \ve_{0},
\end{align*}
and
\begin{align}\label{equ:mass,moment,zero,initial}
\int_{\T^{3}}\int_{\R^{3}}\lrc{F_{0}(x,v)-G(v)}dvdx=0,\quad \int_{\T^{3}}\int_{\R^{3}} v\lrc{F_{0}(x,v)-G(v)}dvdx=0,
\end{align}
then for $0<\al\leq \al_{0}$, the USF Boltzmann equation \eqref{equ:inhomogeneous,Boltzmann,shear} admits a unique solution $F(t,x,v)\geq 0$ satisfying the global stability estimate
\begin{align}\label{equ:global stability,estimate,thm,P0}
\bbn{w_{l}\lrc{e^{3\beta t}F(t,x,e^{\beta t}v)-G(v)}}_{L_{x,v}^{\infty}}\leq C\sum_{k=0,1}
\n{w_{l}\pa_{x}^{k}\lrc{F_{0}(x,v)-G(v)}}_{L_{x,v}^{\infty}}.
\end{align}
For the non-zero frequency mode, we have the exponential decay estimate
\begin{align}\label{equ:global stability,estimate,thm,P1}
\bn{w_{l}e^{3\beta t}P_{\neq 0}^{x}F(t,x,e^{\beta t}v)}_{L_{x,v}^{\infty}}\leq Ce^{-\la_{0}t}\sum_{k=0,1}
\n{w_{l}\pa_{x}^{k}\lrc{F_{0}(x,v)-G(v)}}_{L_{x,v}^{\infty}}.
\end{align}

Furthermore,
the energy moment exhibits an exponential convergence toward a long-time limit state. More precisely, there exists a limit energy state $\mathbf{E}_{\al}(\infty)$ such that
\begin{align}\label{equ:exconvergence,long-time,state}
|\mathbf{E}_{\al}(t)-\mathbf{E}_{\al}(\infty)|\lesssim e^{-\la_{0}t},
\end{align}
where the renormalized energy $\mathbf{E}_{\al}(t)$ is defined by
\begin{align}\label{equ:def,energy moment}
\mathbf{E}_{\al}(t):=e^{-2\beta t}\int_{\T^{3}}\int_{\R^{3}}F(t,x,v)|v|^{2}dvdx-\int_{\R^{3}}G(v)|v|^{2}dv.
\end{align}
In particular, we obtain a sharp characterization of the long-time dynamics of the energy growth
\begin{align}\label{equ:long-time,state,energy,growth}
\bbabs{\int_{\T^{3}}\int_{\R^{3}}F(t,x,v)|v|^{2}dvdx-e^{2 \beta t}\lrs{\int_{\R^{3}}G(v)|v|^{2}dv+\mathbf{E}_{\al}(\infty)}}\lesssim e^{-\la_{0}t/2}.
\end{align}

\end{theorem}
We here present some remarks on Theorem \ref{thm:main theorem}.
\begin{itemize}
\item For $\al=0$, there have been a significant progress in the study of both the homogeneous and inhomogeneous Boltzmann equation. For $\al\neq 0$, to the best of our knowledge, Theorem \ref{thm:main theorem} appears to be the first result on the global stability problem to the inhomogeneous USF Boltzmann equation \eqref{equ:inhomogeneous,Boltzmann,shear}, which provides a sharp characterization of the long-time dynamics of energy growth.
\item The long-time limit state $\mathbf{E}_{\al}(\infty)$ is explicitly computable via formula \eqref{equ:long-time limit state} and effectively approximated via estimate \eqref{equ:exconvergence,energy moment}. Notably, for $\al=0$, we have $\mathbf{E}_{\al}(\infty)=\mathbf{E}_{\al}(0)$ with $G=\mu$, which aligns with the classical energy conservation law for the Boltzmann equation in the absence of shear effects. Further, this result also recovers the stability of the global Maxwellian established in \cite[Proposition 4.3]{AEP87} by Arkeryd, Esposito and Pulvirenti.

\item The homogeneous case is a special case of Theorem \ref{thm:main theorem}. Via formula \eqref{equ:long-time limit state}, one can achieve that the energy $\mathbf{E}_{\al}(\infty)=0$ by imposing a suitable initial condition on the second-order moment. In contrast, the inhomogeneous setting involves a nontrivial interaction between spatial inhomogeneity and collision processes.
     Consequently, the long-time limit state $\mathbf{E}_{\alpha}(\infty)$ does not vanish but is instead intrinsically determined by the global solution itself.
     This intrinsic dependence introduces a substantial difficulty in the analysis of long-time behavior.

\end{itemize}

\subsection{Literatures}
The class of homoenergetic solutions to the Boltzmann equation was introduced independently by Truesdell \cite{Tru56} and Galkin \cite{Gal58} as exact, spatially homogeneous solutions associated with a prescribed linear velocity field. Within this class, the uniform shear flow has been extensively studied, serving both as a mathematically tractable model and as a canonical nonequilibrium state in kinetic theory. See for example the early work \cite{BCS96,Cer89,Cer01,Cer02}.

 Numerically, Montanero-Santos-Garz\'{o} \cite{MSG96} performed direct Monte Carlo simulations of the USF Boltzmann equation, demonstrating strongly anisotropic velocity distributions and non-Newtonian rheological properties. Later, Acedo-Santos-Bobylev \cite{ASB02} provided an analytic description of the emergence of algebraic high-energy tails in the velocity distribution under shear, within the Maxwell molecule framework. These works revealed the rich non-equilibrium structure of USF solutions beyond the near-equilibrium regime.

Significant progress has been made in recent years toward a rigorous theory of homoenergetic solutions. James-Nota-Vel\'{a}zquez  have carried out a comprehensive program of analysis in a series of works. In \cite{JNV19}, they established the existence of self-similar profiles for homoenergetic solutions and analyzed their entropy properties. In subsequent papers, they investigated the long-time asymptotics of these solutions: in the collision-dominated regime \cite{JNV19long} and in the hyperbolic-dominated regime \cite{JNV20}, providing a detailed classification of the asymptotic behaviors depending on the underlying balance between transport and collision effects. Complementing these studies, Bobylev-Nota-Vel\'{a}zquez \cite{BNV20} investigated self-similar asymptotics for a modified Maxwell-Boltzmann equation in systems subject to deformations, offering insight into deformation-driven kinetic dynamics and the approach to self-similarity.

Other works have focused specifically on the USF Boltzmann equation. For Maxwell molecules, Duan-Liu \cite{DL21} rigorously proved the existence, uniqueness, non-negativity, and stability of self-similar profiles under small shear rates. Recent studies have extended these results to more general interaction potentials. Kepka \cite{Kep23} investigated self-similar profiles for non-cutoff Maxwell molecules and examined the long-time behavior in the collision-dominated regime for hard potentials. Independently, Duan-Liu \cite{DL25} studied uniform shear flow for hard potentials using energy and moment methods, without relying on Fourier-transform techniques. Furthermore, the 3D kinetic Couette flow via the Boltzmann equation in the diffusive limit was rigorously analyzed in \cite{DLSY24}, providing a link between kinetic theory and hydrodynamic descriptions.

The stability of shear flows has been extensively investigated for fluid dynamical equations such as the Euler and Navier-Stokes systems; see, for instance, \cite{BGM17,BGM19,BM15,BMV16,IJ20,IJ23,IJ23icm,MZ24,SH01,WZ21,WZ23,WZZ18}. It is therefore of importance to develop analogous mathematical results at the kinetic level. Furthermore, understanding the relationship between these fluid solutions and Boltzmann solutions through a rigorous justification of the hydrodynamic limit under small shear strength presents a key challenge. As a first step toward this long-term goal, it is crucial to investigate the global stability of the sheared Boltzmann equation, particularly in the inhomogeneous setting. In this direction,
 these well-posedness results for the stationary self-similar profile $G(v)$ emerging in uniform shear flow, as established in \cite{BNV20,DL21,JNV19}, provides a necessary foundation for further analysis of long-time behavior.

\subsection{Ideas and Strategies}

Compared to the classical spatially inhomogeneous Boltzmann equation, the USF Boltzmann equation exhibits a distinctive feature of energy growth.
Understanding the long-time behavior of this energy growth is not only a central goal but also a key step toward addressing the global stability of solutions around the self-similar profile $G(v)$.

In the steady case, as discussed in \cite{DL21}, a major difficulty arises from a velocity growth term induced by the shearing motion. For the time-dependent problem in a spatially inhomogeneous setting, further challenges emerge, including unbounded energy growth and the complex interaction between spatial dependence and the nonlinear collision term. Moreover, even in a renormalized formulation, the energy conservation law, which typically serves as a key point for global stability analysis, fails to hold. In the homogeneous setting, one could overcome this difficulty by constructing a normal solution that conserves all physical quantities, including energy, through a scaling that depends on the solution itself, c.f. \cite{DL25,Kep23,Kep24}. However, such a conservative framework becomes substantially more difficult to implement in the presence of spatial inhomogeneity.

To address these difficulties, we propose a novel approach developed within a non-conservative framework.
The key technical contribution lies in deriving a global estimate for the energy moment within this non-conservative setting. This global estimate, combined with Caflisch's decomposition \cite{Caf80} and Guo's $L^{\infty}$--$L^{2}$ energy method \cite{Guo10}, yields a systematic argument for characterizing the dynamics of energy growth.

Our analysis relies on a structural property inherent to Maxwell molecules. Specifically, we combine a low-high frequency decomposition, a spectral analysis of the matrix associated with the second-order moment equation, and a crucial cancellation property in the zero-frequency mode of the nonlinear collision term.

More precisely, since the energy is not conserved in this setting, we first decompose the energy component $c$ into its zero-frequency and non-zero frequency modes:
\begin{align*}
\sup_{s\in[0,t]}\n{c(s)}_{L_{x}^{2}}\leq \sup_{s\in[0,t]}\n{P_{0}^{x}c(s)}_{L_{x}^{2}}+
\sup_{s\in[0,t]}\n{P_{\neq 0}^{x}c(s)}_{L_{x}^{2}}.
\end{align*}
We then derive the evolution equation for the second-order moment, restricted to the zero-frequency mode:
\begin{align}\label{sec-mo-in}
\frac{d}{dt}U+\lrc{2\be I+
\begin{pmatrix}
0&2\al &0\\
0&2b_{0} & \al \\
-\frac{2b_{0}}{3}&0&2b_{0}
\end{pmatrix}}
U=
\begin{pmatrix}
0\\
P_{0}^{x}\lra{Q(\wt{f},\wt{f}),v_{1}v_{2}}\\
P_{0}^{x}\lra{Q(\wt{f},\wt{f}),v_{2}v_{2}}
\end{pmatrix}.
\end{align}
Here,
$
U=[
P_{0}^{x}E,
P_{0}^{x}d_{12},
P_{0}^{x}d_{22}]^T$, in which $[E,d_{12},d_{22}]$ represent the second-order energy moments as defined in \eqref{equ:hydrodynamic moments}. Owing to the absence of any decay in the energy moment equation \eqref{sec-mo-in}, the nonlinear term in \eqref{sec-mo-in} becomes dangerous and could potentially lead to finite-time blow-up.
Drawing upon the analysis of the collision operator's action on quadratic velocity moments by Truesdell and Muncaster \cite[Chapter XII]{TM80},
a null structure is revealed for the zero-frequency mode under consideration:
\begin{align*}
P_{0}^{x}\lra{Q(\wt{f},\wt{f}),v_{1}v_{2}}
=&-2b_{0}P_{0}^{x}\lrc{\lrs{P_{\neq 0}^{x}\lra{\wt{f},v_{1}v_{2}}}
\lrs{P_{\neq 0}^{x}\lra{\wt{f},1}}-\lrs{P_{\neq 0}^{x}\lra{\wt{f},v_{1}}}
\lrs{P_{\neq 0}^{x}\lra{\wt{f},v_{2}}}},
\end{align*}
and
\begin{align*}
&P_{0}^{x}\lra{Q(\wt{f},\wt{f}),v_{2}v_{2}}\\
 =&-2b_{0}P_{0}^{x}\lrc{\lrs{ P_{\neq 0}^{x}\bblra{\wt{f},v_{2}v_{2}-\frac{|v|^{2}}{3}}}
 \lrs{P_{\neq 0}^{x}\lra{\wt{f},1}}+\lrs{P_{\neq 0}^{x}\lra{\wt{f},v_{2}}}^{2}-\frac{1}{3}\sum_{i=1}^{3}\lrs{P_{\neq 0}^{x}\lra{\wt{f},v_{i}}}^{2}}.
\end{align*}
This together with the spectrum analysis of the linear equation of \eqref{sec-mo-in} further gives a global estimate
\begin{align*}%\label{equ:L2,U,estimate}
\sup_{s\in[0,t]}\n{U(s)}_{L_{x}^{2}}\leq \n{U(0)}_{L_{x}^{2}}+\lrs{\sup_{s\in[0,t]}e^{\la_{0} s}\n{w_{l}\pa_{x}g_{1}(s)}_{L_{x,v}^{\infty}}}^{2}+\lrs{\sup_{s\in[0,t]}e^{\la_{0} s}\n{\pa_{x}g_{2}(s)}_{L_{x,v}^{2}}}^{2}.
\end{align*}
The cancellation property we uncover is closely analogous to the null condition in nonlinear wave equations, first identified by Klainerman \cite{Kla84,Kla86} and Christodoulou \cite{Chr86}, which suppresses the most dangerous resonant interactions and ensures global existence for small initial data. In our kinetic setting, the zero-frequency cancellation similarly prevents resonant energy growth and allows global estimates.
We believe this mechanism is particularly significant for shear flow in the Boltzmann equation, where the long-time behavior of energy moments is central to understanding the interplay between kinetic effects and macroscopic shear.
Our result uncovers an intrinsic stability feature of shear-driven kinetic systems, suggesting that such null structures may persist under perturbations such as non-cutoff kernels and more general intermolecular potentials. We expect this structural insight to provide a natural foundation for future studies on transition thresholds and stability problems for shear flows in the Boltzmann framework.

The paper is structured as follows. In Section \ref{sec:Local Well-posedness}, we introduce a renormalized perturbation equation around the self-similar profile and establish local well-posedness, which serves as the foundation for the subsequent global stability analysis. Section \ref{sec:Weighted Estimates} is devoted to weighted $L^{\infty}$ estimates, reducing the stability problem to an $L^{2}$ energy estimate. In Section \ref{sec:L2 Energy Estimates}, we carry out a detailed $L^{2}$ energy analysis, subdivided into lower-order energy estimates in Subsection \ref{sec:Lower Energy Estimates}, second-order energy moment estimates in Subsection \ref{sec:Moment Estimates}, and higher-order energy estimates in Subsection \ref{sec:Higher Energy Estimates}. Finally, in Section \ref{sec:Global Stability}, we conclude the proof of Theorem \ref{thm:main theorem}.

\section{Local Well-posedness}\label{sec:Local Well-posedness}
This section is devoted to establishing the local well-posedness for the Cauchy problem \eqref{equ:inhomogeneous,Boltzmann,shear}, which is a necessary first step towards the subsequent global stability analysis. By the self-similar scaling structure in \eqref{equ:self-similar solution}, we introduce the self-similar scaling transformation
\begin{align*}
f(t,x,v)=e^{3\beta t}F(t,x,e^{\beta t}v).
\end{align*}
Under this scaling, the Cauchy problem \eqref{equ:inhomogeneous,Boltzmann,shear} reduces to
\begin{equation}\label{equ:inhomogeneous,Boltzmann,shear,renormalized}
\left\{
\begin{aligned}
&\pa_{t}f+e^{\beta t}\sum_{i=1}^{3}v_{i}\pa_{x_{i}}f-\beta \nabla_{v}(vf)-\al v_{2}\pa_{v_{1}}f
=Q(f,f),\\
&f(0,x,v)=F_{0}(x,v).
\end{aligned}
\right.\end{equation}
As indicated in \eqref{equ:self-similar solution}, the solution $F(t, x, v)$ exhibits the asymptotic behavior
\[ F(t, x, v) \sim e^{-3\beta t} G\big(e^{-\beta t}v\big) \quad \text{as} \quad t \to \infty,\]
where $G(v)$ is the self-similar profile. This asymptotic relation implies that the rescaled distribution function $f(t,x,v)$ exhibits the stability property around the self-similar profile $G(v)$ in the long-time limit.

To investigate the stability of this profile, we introduce the perturbation
\begin{equation*}
    \widetilde{f}(t, x, v) := f(t, x, v) - G(v).
\end{equation*}
Combining \eqref{equ:inhomogeneous,Boltzmann,shear,renormalized} and \eqref{equ:steady profile,G}, a straightforward calculation shows that $\tilde{f}$ satisfies the perturbation equation
\begin{equation}\label{equ:perturbation equation,wf}
\left\{
\begin{aligned}
&\pa_{t}\wt{f}+e^{\beta t}\sum_{i=1}^{3}v_{i}\pa_{x_{i}}\wt{f}-\beta \nabla_{v}(v\wt{f})-\al v_{2}\pa_{v_{1}}\wt{f}
=Q(\wt{f},\wt{f})+Q_{sym}(\wt{f},G),\\
&\wt{f}(0,x,v)=F_{0}(x,v)-G(v),
\end{aligned}
\right.
\end{equation}
where $Q_{sym}(f_{1},f_{2}):=Q(f_{1},f_{2})+Q(f_{2},f_{1})$ is the symmetrized collision operator.

Next, we rescale the perturbation by the global Maxwellian $\mu(v)$ and set $\wt{g}=\mu^{-\frac{1}{2}}\wt{f}$.
Recalling the asymptotic expansion \eqref{equ:G,profile,1,order} that $G=\mu+\al \mu^{\frac{1}{2}}G_{1}$ with $\mu^{\frac{1}{2}}G_{1}\sim \frac{2}{b_{0}}v_{1}v_{2}\mu$, and
substituting this into \eqref{equ:perturbation equation,wf}, we obtain
\begin{equation}\label{equ:perturbation equation,wg}
\left\{
\begin{aligned}
&\pa_t \wt{g}+e^{\beta t}\sum_{i=1}^{3}v_{i}\pa_{x_{i}}\wt{g}-\beta \nabla_{v} \cdot(v \wt{g})+\frac{\beta}{2}|v|^2 \wt{g}-\alpha v_{2} \pa_{v_{1}} \wt{g}
+\frac{\alpha}{2} v_{1} v_{2} \wt{g}+L \wt{g}\\
=&\Gamma(\wt{g}, \wt{g})+\al \Gamma_{sym}\lrs{G_{1}, \wt{g}}, \\
&\wt{g}(0, x, v)=\frac{F_0(x, v)-G(v)}{\sqrt{\mu}}\stackrel{\text { def }}{=} \wt{g}_{0}(x, v),
\end{aligned}
\right.
\end{equation}
where the linear and nonlinear operators are defined by
\begin{align}
L \wt{g}=:&-\mu^{-1 / 2}Q_{sym}(\mu, \sqrt{\mu} \wt{g}),\label{equ:notation,Lg}\\
\Gamma(f_{1}, f_{2})=:&\mu^{-1 / 2} Q(\sqrt{\mu} f_{1}, \sqrt{\mu} f_{2}),\quad
\Gamma_{sym}(f_{1}, f_{2})=:\mu^{-1 / 2} Q_{sym}(\sqrt{\mu} f_{1}, \sqrt{\mu} f_{2}).\label{equ:notation,Gamma}
\end{align}

To analyze \eqref{equ:perturbation equation,wg}, we employ Caflisch's decomposition
\begin{align*}
 \wt{g}=\mu^{-\frac{1}{2}}g_{1}+ g_{2},
\end{align*}
where $g_{1}$ and $g_{2}$ satisfy the coupled system
\begin{equation}\label{equ:g1,equation}
\left\{
\begin{aligned}
&\pa_t g_{1}+e^{\beta t} \sum_{i=1}^{3}v_{i}\pa_{x_{i}} g_{1}-\beta \nabla_{v} \cdot\left(v g_{1}\right)-\alpha v_{2} \pa_{v_{1}} g_{1}+v_0 g_{1} \\
=&-\frac{\beta}{2}|v|^2 \mu^{\frac{1}{2}} g_{2}-\frac{\alpha}{2} \mu^{\frac{1}{2}} v_{1} v_{2} g_{2}+(1-\chi_M) K_{\mu}g_{1} +H\left(g_{1}, g_{2}\right),\\
&g_{1}(0, x, v)=\wt{f}_{0}(x, v)=F_0(x, v)-G(v),
\end{aligned}
\right.
\end{equation}
and
\begin{equation}\label{equ:g2,equation}
\left\{
\begin{aligned}
&\pa_t g_{2}+e^{\beta t} \sum_{i=1}^{3}v_{i}\pa_{x_{i}} g_{2}-\beta \nabla_{v} \cdot\left(v g_{2}\right)-\alpha v_{2} \pa_{v_{1}} g_{2}+L g_{2} =\mu^{-1 / 2}\chi_{M} K_{\mu} g_{1},  \\
&g_{2}(0, x, v) =0,
\end{aligned}\right.
\end{equation}
with the cutoff function $\chi_{M}(v)=1_{\lr{|v|\leq M}}$.
Here, the operators $L$, $K$, $K_{\mu}$, and $H$ are given by
\begin{align}
&L f=\nu_{0} f-K f,\ \nu_{0}=\int_{\mathbb{R}^3} \int_{\mathbb{S}^2} B_0(\cos \theta) \mu\left(v_*\right) d \omega dv_*,\label{equ:notation,L,K}\\
 & Kf=\mu^{-\frac{1}{2}}\lrc{Q(\mu^{\frac{1}{2}}f,\mu)+Q^{+}(\mu, \mu^{\frac{1}{2}} f)},
\label{equ:notation,K}\\
& K_{\mu}(f)=\mu^{\frac{1}{2}}K(\mu^{-\frac{1}{2}}f)=Q(f,\mu)+Q^{+}(\mu,f),
\label{equ:notation,Kmu}\\
&H\left(g_{1}, g_{2}\right)= Q(g_{1}+\mu^{\frac{1}{2}} g_{2}, g_{1}+\mu^{\frac{1}{2}} g_{2})+\al Q_{sym}(g_{1}+\mu^{\frac{1}{2}} g_{2}, \mu^{\frac{1}{2}}G_{1}) .\label{equ:notation,H}
\end{align}

We now proceed to establish the local well-posedness theory for the coupled system \eqref{equ:g1,equation}--\eqref{equ:g2,equation}.

\begin{lemma}[Local well-posedness]\label{lemma:lwp,g12}
 For initial data $[g_{1}(0),g_{2}(0)]$ satisfying
\begin{align*}
\n{\left[g_{1}(0), g_{2}(0)\right]}_{Y}:=\sum_{k=0,1} \left\{\n{w_{l} \pa_{x}^{k} g_{1}(0)}_{L_{x,v}^{\infty}}+\n{w_{l} \pa_{x}^{k} g_{2}(0)}_{L_{x,v}^{\infty}}\right\}
<\infty,
\end{align*}
there exists a positive time
$$T\sim \n{\left[g_{1}(0), g_{2}(0)\right]}_{Y}^{-1},$$
 such that the coupled system \eqref{equ:g1,equation}--\eqref{equ:g2,equation} admits a unique local-in-time solution $\left[g_{1}, g_{2}\right]$ satisfying
\begin{align*}
\left\|\left[g_{1}, g_{2}\right]\right\|_{\ol{Y}_{T}}\leq 2C\n{[g_{1}(0),g_{2}(0)]}_{Y},
\end{align*}
where the norm is defined by
\begin{equation*}
\n{\left[\mathcal{G}_1, \mathcal{G}_2\right]}_{\overline{Y}_{T}}=:\sum_{k=0,1} \lrc{\sup_{t\in[0,T]}\left\|w_{l} \pa_{x}^{k} \mathcal{G}_1(t)\right\|_{L_{x,v}^{\infty}}+\sup_{t\in[0,T]}\left\|w_{l} \pa_{x}^{k} \mathcal{G}_2(t)\right\|_{L_{x,v}^{\infty}}} .
\end{equation*}
\end{lemma}
\begin{proof}
The proof relies on the Duhamel principle and the contraction mapping method. Applying Duhamel's formula, we reformulate the coupled system \eqref{equ:g1,equation}--\eqref{equ:g2,equation} in integral form through a nonlinear operator $\mathcal{N}$:
\begin{align*}
[g_{1},g_{2}]=\mathcal{N}[g_{1},g_{2}].
\end{align*}
More precisely,
let $[h_{1},h_{2}]=[w_{l} \pa_{x}^{k} g_{1},w_{l} \pa_{x}^{k} g_{2}]$ for $k=0,1$. Then $[h_{1},h_{2}]$ satisfies
\begin{equation}\label{equ:h1,equation}
\left\{\begin{aligned}
&\partial_t h_1+e^{\beta t} \sum_{i=1}^{3}v_{i}\pa_{x_{i}} h_1-\beta \nabla_{v} \cdot\left(v h_1\right)-\alpha v_{2} \partial_{v_{1}} h_1+2 l \beta \frac{|v|^2}{1+|v|^2} h_1 +2 l \alpha \frac{v_{2} v_{1}}{1+|v|^2} h_1+\nu_{0} h_1 \\
 =&-\frac{\beta}{2}|v|^2 \mu^{\frac{1}{2}} h_2-\frac{\alpha}{2} \mu^{\frac{1}{2}} v_{1} v_{2} h_2+
  (1-\chi_M) w_{l} K_{\mu}\left(\frac{h_1}{w_{l}}\right)+w_{l} \pa_{x}^{k} H\left(g_{1}, g_{2}\right), \\
&h_1(0, x, v)=w_{l} \pa_{x}^{k} g_{1}(0,x, v),
\end{aligned}\right.
\end{equation}
\begin{equation}\label{equ:h2,equation}
\left\{
\begin{aligned}
&\partial_t h_2+e^{\beta t} \sum_{i=1}^{3}v_{i}\pa_{x_{i}} h_2-\beta \nabla_{v} \cdot\left(v h_2\right)-\alpha v_{2} \partial_{v_{1}} h_2+2 l \beta \frac{|v|^2}{1+|v|^2} h_2 +2 l \alpha \frac{v_{2} v_{1}}{1+|v|^2} h_2+\nu_{0} h_2\\
=&w_{l} \mu^{-1 / 2}\chi_{M} K_{\mu}\left(\frac{h_1}{w_{l}}\right)+w_{l} K\left(\frac{h_2}{w_{l}}\right), \\
&h_2(0, x, v)=w_{l} \pa_{x}^{k} g_{2}(0,x, v).
\end{aligned}\right.
\end{equation}

We proceed by defining the characteristic trajectory $[s, X(s ; t, x, v), V(s ; t, x, v)]$ for systems \eqref{equ:h1,equation} and \eqref{equ:h2,equation} passing through $(t,x,v)$, governed by the following ODE system
\begin{equation}\label{equ:characteristic ODE}
\left\{\begin{aligned}
&\frac{d X_{i}}{d s}=e^{\beta s} V_{i}(s ; t, x, v), \quad i=1,2,3,\\
&\frac{d V_1}{d s}=-\beta V_1(s ; t, x, v)-\alpha V_2(s ; t, x, v), \\
&\frac{d V_i}{d s}=-\beta V_i(s ; t, x, v), \quad i=2,3, \\
&X(t ; t, x, v)=x, \quad V(t ; t, x, v)=v.
\end{aligned}\right.
\end{equation}
The explicit solution of this system is given by
\begin{equation}\label{equ:characteristic ODE,solution}
\left\{\begin{aligned}
X_{1}(s ; t, x, v) & =e^{\beta t}\left(e^{-\beta t}x_{1}-(t-s) v_{1}-\frac{1}{2} \alpha(t-s)^2 v_{2}\right), \\
X_{i}(s ; t, x, v) & =e^{\beta t}\left(e^{-\beta t}x_{i}-(t-s) v_i\right),\quad i=2,3, \\
V_1(s ; t, x, v) & =e^{\beta(t-s)}\left(v_{1}+\alpha v_{2}(t-s)\right), \\
V_i(s ; t, x, v) & =e^{\beta(t-s)} v_i ,\quad i=2,3.
\end{aligned}\right.
\end{equation}
Utilizing Duhamel's formula, we express the solution as
\begin{align}\label{equ:duhamel,h1,h2}
[h_{1},h_{2}]=[\mathcal{N}_{1}(g_{1},g_{2}),\mathcal{N}_{2}(g_{1},g_{2})],
\end{align}
where
\begin{align*}
\mathcal{N}_{1}(g_{1},g_{2})=&I_{1,1}+I_{1,2}+I_{1,3}+I_{1,4}+I_{1,5},\\
\mathcal{N}_2\left(g_{1}, g_{2}\right)=&I_{2,1}+I_{2,2}+I_{2,3}.
\end{align*}
The integral components are defined as follows:
\begin{equation}\label{equ:items,I1,i}
\left\{
\begin{aligned}
I_{1,1}=&  e^{-\int_0^t \mathcal{A}(\tau, V(\tau)) d  \tau}\left(w_{l} \pa_{x}^{k}g_{1}\right)(0,X(0), V(0)),\\
I_{1,2}=&-\frac{\beta}{2} \int_0^t e^{-\int_s^t \mathcal{A}(\tau, V(\tau)) d  \tau}|V(s)|^2 \sqrt{\mu}(V(s)) h_2(s,X(s),V(s)) d  s,\\
I_{1,3}=& -\alpha \int_0^t e^{-\int_s^t \mathcal{A}(\tau, V(\tau)) d  \tau} \frac{V_1(s) V_2(s)}{2} \sqrt{\mu}(V(s)) h_2(s,X(s),V(s)) d  s, \\
I_{1,4}=&\int_0^t e^{-\int_s^t \mathcal{A}(\tau, V(\tau)) d  \tau}\lrc{(1-\chi_M) w_{l} K_{\mu}\left(\frac{h_1}{w_{l}}\right)}(s,X(s),V(s)) d  s ,\\
I_{1,5}=& \int_0^t e^{-\int_s^t \mathcal{A}(\tau, V(\tau)) d  \tau}\left(w_{l} \pa_{x}^{k} H(g_{1},g_{2})\right)(s,X(s),V(s)) d  s,
\end{aligned}
\right.
\end{equation}
and
\begin{equation}\label{equ:items,I2,i}
\left\{
\begin{aligned}
I_{2,1}= & e^{-\int_0^t \mathcal{A}(\tau, V(\tau)) d  \tau}\left(w_{l} \pa_{x}^{k}g_{2}\right)(0,X(0), V(0)), \\
I_{2,2}=& \int_0^t e^{-\int_s^t \mathcal{A}(\tau, V(\tau)) d  \tau}\lrc{\chi_{M} \mu^{-\frac{1}{2}} w_{l} K_{\mu}\left(\frac{h_1}{w_{l}}\right)}(s,X(s),V(s)) d  s,\\
I_{2,3}=&\int_0^t e^{-\int_s^t \mathcal{A}(\tau, V(\tau)) d  \tau}\left[w_{l} K\left(\frac{h_2}{w_{l}}\right)\right](s,X(s),V(s)) d  s,
\end{aligned}
\right.
\end{equation}
where the damping coefficient satisfies the lower bound:
\begin{equation}\label{equ:A(t,V),lower bound}
\mathcal{A}(\tau, V(\tau))=\nu_{0}-3 \beta+2 l \beta \frac{|V(\tau)|^2}{1+|V(\tau)|^2}+2 l \alpha \frac{V_2(\tau) V_1(\tau)}{1+|V(\tau)|^2} \geq \frac{\nu_{0}}{2}.
\end{equation}

We then define the set
\begin{align*}
B_{T}=\lr{\lrc{\mathcal{G}_1, \mathcal{G}_2}:\n{\left[\mathcal{G}_1, \mathcal{G}_2\right]}_{\overline{Y}_{T}}\leq 2C\n{[g_{1}(0),g_{2}(0)]}_{Y}
},
\end{align*}
and now demonstrate that the nonlinear map
$\mathcal{N}[\cdot,\cdot]$ is a contraction mapping on $B_{T}$ for sufficiently small $T>0$.

We begin by estimating the norm
\begin{align*}
\n{\mathcal{N}[g_{1},g_{2}]}_{\ol{Y}_{T}}=&
\sum_{k=0,1}\sup_{t\in[0,T]}\n{[h_{1},h_{2}](t)}_{L_{x,v}^{\infty}}\\
\leq&\sum_{k=0,1}\lrc{\sup_{t\in[0,T]}
\n{\mathcal{N}_{1}[g_{1},g_{2}](t)}_{L_{x,v}^{\infty}}+\sup_{t\in[0,T]}\n{\mathcal{N}_{2}[g_{1},g_{2}](t)}_{L_{x,v}^{\infty}}}\\
\leq& \sum_{k=0,1}\lrc{\sum_{i=1}^{5}\sup_{t\in[0,T]}\n{I_{1,i}(t)}_{L_{x,v}^{\infty}}+\sum_{i=1}^{3}
\sup_{t\in[0,T]}\n{I_{2,i}(t)}_{L_{x,v}^{\infty}}}.
\end{align*}

For $I_{1,1}$, $I_{1,2}$, and $I_{1,3}$, we have
\begin{align}\label{equ:lwp,I1,123}
\sum_{k=0,1}\sum_{i=1}^{3}\n{I_{1,i}(t)}_{L_{x,v}^{\infty}}\leq \sum_{k=0,1}\lrc{\n{w_{l}\pa_{x}^{k}g_{1}(0)}_{L_{x,v}^{\infty}}+C\al T\sup_{t\in [0,T]}\n{h_{2}(t)}_{L_{x,v}^{\infty}}}.
\end{align}

For $I_{1,4}$, we use Lemma \ref{lemma:large velocity,K} to obtain
\begin{align}\label{equ:lwp,I1,4}
\sum_{k=0,1}\n{I_{1,4}(t)}_{L_{x,v}^{\infty}}\leq &T\sum_{k=0,1}\sup_{t\in[0,T]}\bbn{(1-\chi_{M})w_{l}K_{\mu}\lrs{\frac{h_{1}}{w_{l}}}}_{L_{x,v}^{\infty}}\\
\leq& \frac{CT}{l}\sum_{k=0,1}\sup_{t\in[0,T]}\n{h_{1}(t)}_{L_{x,v}^{\infty}}.\notag
\end{align}

For $I_{1,5}$, we derive
\begin{align*}
\sum_{k=0,1}\n{I_{1,5}(t)}_{L_{x,v}^{\infty}}\leq &T\sum_{k=0,1}\sup_{t\in[0,T]}\n{w_{l} \pa_{x}^{k} H(g_{1},g_{2})(t)}_{L_{x,v}^{\infty}}.
\end{align*}
Recalling the definition of $H(g_{1},g_{2})$ in \eqref{equ:notation,H}, we use Lemma \ref{lemma:Q,bilinear estimate,Lp} to obtain
\begin{align*}
&\n{w_{l} \pa_{x}^{k}H(g_{1},g_{2})}_{L_{x,v}^{\infty}}\\
\leq& \n{w_{l}\pa_{x}^{k}Q(g_{1}+\mu^{\frac{1}{2}} g_{2}, g_{1}+\mu^{\frac{1}{2}} g_{2})}_{L_{x,v}^{\infty}}+\al
\n{w_{l}\pa_{x}^{k}Q_{sym}(g_{1}+\mu^{\frac{1}{2}} g_{2}, \mu^{\frac{1}{2}}G_{1})}_{L_{x,v}^{\infty}}\\
\lesssim& \sum_{k=0,1}\lrs{\n{w_{l}\pa_{x}^{k}g_{1}}_{L_{x,v}^{\infty}}+\n{w_{l}\pa_{x}^{k}g_{2}}_{L_{x,v}^{\infty}}}^{2}+\al\lrs{\n{w_{l}\pa_{x}^{k}g_{1}}_{L_{x,v}^{\infty}}
+\n{w_{l}\pa_{x}^{k}g_{2}}_{L_{x,v}^{\infty}}},
\end{align*}
where in the last inequality we have also used $\n{w_{l}\mu^{\frac{1}{2}}G_{1}}_{L_{v}^{\infty}}\lesssim 1$, as indicated in \eqref{equ:G,profile,velocity}.

Proceeding analogously for $I_{2,i}$, we deduce
\begin{align*}
&\sum_{k=0,1}\sum_{i=1}^{3}\n{I_{2,i}(t)}_{L_{x,v}^{\infty}}\\
\leq& \sum_{k=0,1}\lrc{\n{w_{l}\pa_{x}^{k}g_{2}(0)}_{L_{x,v}^{\infty}}+C T\sup_{t\in [0,T]}\n{h_{1}(t)}_{L_{x,v}^{\infty}}+C T\sup_{t\in [0,T]}\n{h_{2}(t)}_{L_{x,v}^{\infty}}}.
\end{align*}

Combining the estimates for $I_{1,i}$ and $I_{2,i}$, we arrive at
\begin{align*}
\n{\mathcal{N}[g_{1},g_{2}]}_{\ol{Y}_{T}}\leq C\lrc{\n{[g_{1}(0),g_{2}(0)]}_{Y}+\lrs{\al+\frac{1}{l}+1}T\n{[g_{1},g_{2}]}_{\ol{Y}_{T}}+
T\n{[g_{1},g_{2}]}_{\ol{Y}_{T}}^{2}}.
\end{align*}
By choosing $T \sim\frac{1}{8}\min\lr{1,(C\n{[g_{1}(0),g_{2}(0)]}_{Y})^{-1}}$, we conclude
\begin{align*}
\n{\mathcal{N}[g_{1},g_{2}]}_{\ol{Y}_{T}}\leq 2C\n{[g_{1}(0),g_{2}(0)]}_{Y}.
\end{align*}
Thus, $\mathcal{N}$ maps the set $B_{T}$ into itself.

Now, we consider two arbitrary elements $[g_{1}',g_{2}']\in B_{T}$ and $[g_{1}'',g_{2}'']\in B_{T}$. We estimate the difference between their images under $\mathcal{N}$ as follows
\begin{align*}
&\n{\mathcal{N}[g_{1}',g_{2}']-\mathcal{N}[g_{1}'',g_{2}'']}_{\ol{Y}_{T}}\\
\leq&\sum_{k=0,1}\sup_{t\in[0,T]}\n{h_{1}'(t)-h_{1}''(t)}_{L_{x,v}^{\infty}}+\sum_{k=0,1}\sup_{t\in[0,T]}\n{h_{2}'(t)-h_{2}''(t)}_{L_{x,v}^{\infty}}\\
\leq&C\lrs{\al+\frac{1}{l}+1}T\lrs{\n{g_{1}'-g_{1}''}_{\ol{Y}_{T}}+
\n{g_{2}'-g_{2}''}_{\ol{Y}_{T}}}\\
&+CT\lrs{\n{[g_{1}',g_{2}']}_{\ol{Y}_{T}}+\n{[g_{1}'',g_{2}'']}_{\ol{Y}_{T}}}\lrs{\n{g_{1}'-g_{1}''}_{\ol{Y}_{T}}+
\n{g_{2}'-g_{2}''}_{\ol{Y}_{T}}}\\
\leq& \frac{1}{2}\lrs{\n{g_{1}'-g_{1}''}_{\ol{Y}_{T}}+
\n{g_{2}'-g_{2}''}_{\ol{Y}_{T}}}.
\end{align*}
This shows that
 $\mathcal{N}$ is a contraction map on $B_{T}$ and admits a unique fixed point $[g_{1},g_{2}]\in B_{T}$, provided the time scale satisfies $T \sim\frac{1}{8}\min\lr{1,(C\n{[g_{1}(0),g_{2}(0)]}_{Y})^{-1}}$. The proof of Lemma \ref{lemma:lwp,g12} is then complete.
\end{proof}

\section{Weighted $L_{x,v}^{\infty}$ Estimates}\label{sec:Weighted Estimates}
In the section, we derive weighted $L_{x,v}^{\infty}$ estimates for solutions to the coupled system \eqref{equ:g1,equation}--\eqref{equ:g2,equation}.
For notational simplicity, we introduce the following weighted $L_{x,v}^{\infty}$ norms
\begin{align*}
\n{g_{i}}_{\ol{X}_{0,t}}:=&\sup_{s\in[0,t]}\n{w_{l}g_{i}(s)}_{L_{x,v}^{\infty}},\quad
\n{g_{i}}_{\ol{X}_{1,t}}:=\sup_{s\in[0,t]}e^{\la_{0} s}\n{w_{l}\pa_{x}g_{i}(s)}_{L_{x,v}^{\infty}},\\
\n{[g_{1},g_{2}]}_{\ol{X}_{k,t}}:=&\n{g_{1}}_{\ol{X}_{k,t}}+\n{g_{2}}_{\ol{X}_{k,t}},\quad k=0,1,
\end{align*}
where $\la_{0}>0$ represents the decay rate associated with the non-zero frequency modes.

The main result of this section provides a priori estimates for solutions to the coupled system \eqref{equ:g1,equation}--\eqref{equ:g2,equation}.
\begin{lemma}\label{lemma,Linf,estimate}
Let $[g_{1},g_{2}]$ be the solution to \eqref{equ:g1,equation}--\eqref{equ:g2,equation}. Then the following estimates hold
\begin{align}
\n{g_{1}}_{\ol{X}_{0,t}}\lesssim &\n{w_{l}\wt{f}_{0}}_{L_{x,v}^{\infty}}+\al \n{g_{2}}_{\ol{X}_{0,t}}+\lrs{\n{g_{1}}_{\ol{X}_{0,t}}+\n{g_{2}}_{\ol{X}_{0,t}}}
\lrs{\n{g_{1}}_{\ol{X}_{0,t}}+\n{g_{2}}_{\ol{X}_{0,t}}},\label{equ:Linf,g1,g2,0}\\
\n{g_{1}}_{\ol{X}_{1,t}}\lesssim &\n{w_{l}\pa_{x}\wt{f}_{0}}_{L_{x,v}^{\infty}}+ \al \n{g_{2}}_{\ol{X}_{1,t}}+\lrs{\n{g_{1}}_{\ol{X}_{0,t}}+\n{g_{2}}_{\ol{X}_{0,t}}}
\lrs{\n{g_{1}}_{\ol{X}_{1,t}}+\n{g_{2}}_{\ol{X}_{1,t}}},\label{equ:Linf,g1,g2,1}\\
\n{g_{2}}_{\ol{X}_{0,t}}\lesssim& \n{g_{1}}_{\ol{X}_{0,t}}+ \sup_{s\in[0,t]} \n{g_{2}(s)}_{L_{x,v}^{2}},\label{equ:Linf,g2,L2,0}\\
\n{g_{2}}_{\ol{X}_{1,t}}\lesssim& \n{g_{1}}_{\ol{X}_{1,t}}+ \sup_{s\in[0,t]} e^{\la_{0} s}\n{\pa_{x}g_{2}(s)}_{L_{x,v}^{2}}.\label{equ:Linf,g2,L2,1}
\end{align}
In particular, we obtain
\begin{align}\label{equ:Linf,L2,estimate}
\sum_{k=0,1}\n{[g_{1},g_{2}]}_{\ol{X}_{k,t}}
\lesssim&\sum_{k=0,1}\n{w_{l} \pa_{x}^{k} \wt{f}_{0}}_{L_{x,v}^{\infty}}+
\sup_{s\in[0,t]}\n{g_{2}(s)}_{L_{x,v}^{2}}+\sup_{s\in[0,t]} e^{\la_{0} s}\n{\pa_{x}g_{2}(s)}_{L_{x,v}^{2}}\\
&+\lrs{\sum_{k=0,1}\n{[g_{1},g_{2}]}_{\ol{X}_{k,t}}}^{2}.\notag
\end{align}
\end{lemma}
\begin{proof}
Recalling the expressions \eqref{equ:duhamel,h1,h2}, \eqref{equ:items,I1,i}, and \eqref{equ:items,I2,i}, we have
\begin{align*}
e^{\la_{0}t}|w_{l} \pa_{x}^{k} g_{1}|\leq e^{\la_{0}t}\sum_{i=1}^{5}\abs{I_{1,i}},
\end{align*}
where
\begin{equation*}
\left\{
\begin{aligned}
I_{1,1}=&  e^{-\int_0^t \mathcal{A}(\tau, V(\tau)) d  \tau}\left(w_{l} \pa_{x}^{k}g_{1}\right)(0,X(0), V(0)),\\
I_{1,2}=&-\frac{\beta}{2} \int_0^t e^{-\int_s^t \mathcal{A}(\tau, V(\tau)) d  \tau}|V(s)|^2 \sqrt{\mu}(V(s)) h_2(s,X(s),V(s)) d  s,\\
I_{1,3}=& -\alpha \int_0^t e^{-\int_s^t \mathcal{A}(\tau, V(\tau)) d  \tau} \frac{V_1(s) V_2(s)}{2} \sqrt{\mu}(V(s)) h_2(s,X(s),V(s)) d  s, \\
I_{1,4}=&\int_0^t e^{-\int_s^t \mathcal{A}(\tau, V(\tau)) d  \tau}\lrc{(1-\chi_M) w_{l} K_{\mu}\left(\pa_{x}^{k}g_{1}\right)}(s,X(s),V(s)) d  s ,\\
I_{1,5}=& \int_0^t e^{-\int_s^t \mathcal{A}(\tau, V(\tau)) d  \tau}\left(w_{l} \pa_{x}^{k} H(g_{1},g_{2})\right)(s,X(s),V(s)) d  s,
\end{aligned}
\right.
\end{equation*}
The coercivity estimate \eqref{equ:A(t,V),lower bound} yields that
\begin{align*}
\mathcal{A}(\tau, V(\tau))\geq \frac{\nu_{0}}{2},\quad e^{-\int_{s}^{t} \mathcal{A}(\tau, V(\tau)) d  \tau}\leq e^{-\frac{\nu_{0}}{2}(t-s)}.
\end{align*}
Following the argument employed in \eqref{equ:lwp,I1,123}--\eqref{equ:lwp,I1,4}, for the first four terms, we derive
\begin{align*}
e^{\la_{0}t}\sum_{i=1}^{4}\abs{I_{1,i}}\leq&\n{w_{l}\pa_{x}^{k}\wt{f}_{0}}_{L_{x,v}^{\infty}}+C \al \lrs{\int_{0}^{t}e^{-\frac{\nu_{0}}{2}(t-s)}e^{\la_{0}(t-s)}ds} \lrs{\sup_{s\in[0,t]}e^{\la_{0}s}\n{w_{l} \pa_{x}^{k} g_{2}(s)}_{L_{x,v}^{\infty}}}\\
&+ \frac{C}{l} \lrs{\int_{0}^{t}e^{-\frac{\nu_{0}}{2}(t-s)}e^{\la_{0}(t-s)}ds} \lrs{\sup_{s\in[0,t]}e^{\la_{0}s}\n{w_{l} \pa_{x}^{k} g_{1}(s)}_{L_{x,v}^{\infty}}}\\
\leq&\n{w_{l}\pa_{x}^{k}\wt{f}_{0}}_{L_{x,v}^{\infty}}+\frac{C}{l}\n{g_{1}}_{\ol{X}_{k,t}}+C\al\n{g_{2}}_{\ol{X}_{k,t}}.
\end{align*}

For the nonlinear term $I_{1,5}$, we obtain
\begin{align*}
e^{\la_{0}t}|I_{1,5}|\leq&  \lrs{\int_{0}^{t}e^{-\frac{\nu_{0}}{2}(t-s)}e^{\la_{0}(t-s)}ds} \lrs{\sup_{s\in[0,t]}e^{\la_{0}s}\n{w_{l} \pa_{x}^{k}H(g_{1},g_{2})}_{L_{x,v}^{\infty}}}\\
\leq&C\lrs{\n{g_{1}}_{\ol{X}_{0,t}}+\n{g_{2}}_{\ol{X}_{0,t}}+\al}
\lrs{\n{g_{1}}_{\ol{X}_{k,t}}+\n{g_{2}}_{\ol{X}_{k,t}}},
\end{align*}
where in the last inequality we have used Lemma \ref{lemma:Q,bilinear estimate,Lp} to get
\begin{align*}
\n{w_{l} \pa_{x}^{k}H(g_{1},g_{2})}_{L_{x,v}^{\infty}}
\leq& \lrs{\n{w_{l}\pa_{x}^{k}g_{1}}_{L_{x,v}^{\infty}}+
 \n{w_{l}\pa_{x}^{k}g_{2}}_{L_{x,v}^{\infty}}}\lrs{\n{w_{l}g_{1}}_{L_{x,v}^{\infty}}+
 \n{w_{l}g_{2}}_{L_{x,v}^{\infty}}}\\
 &+\al\lrs{\n{w_{l}\pa_{x}^{k}g_{1}}_{L_{x,v}^{\infty}}
+\n{w_{l}\pa_{x}^{k}g_{2}}_{L_{x,v}^{\infty}}}.
\end{align*}

Putting together estimates for $I_{1,i}$, and absorbing $(\frac{1}{l}+\al)\n{g_{1}}_{\ol{X}_{k,t}}$, we arrive at
\begin{align*}
\n{g_{1}}_{\ol{X}_{0,t}}\lesssim &\n{w_{l}\wt{f}_{0}}_{L_{x,v}^{\infty}}+\al\n{g_{2}}_{\ol{X}_{0,t}}+\lrs{\n{g_{1}}_{\ol{X}_{0,t}}+\n{g_{2}}_{\ol{X}_{0,t}}}
\lrs{\n{g_{1}}_{\ol{X}_{0,t}}+\n{g_{2}}_{\ol{X}_{0,t}}},\\
\n{g_{1}}_{\ol{X}_{1,t}}\lesssim & \n{w_{l}\pa_{x}\wt{f}_{0}}_{L_{x,v}^{\infty}}+\al\n{g_{2}}_{\ol{X}_{1,t}}+\lrs{\n{g_{1}}_{\ol{X}_{0,t}}+\n{g_{2}}_{\ol{X}_{0,t}}}
\lrs{\n{g_{1}}_{\ol{X}_{1,t}}+\n{g_{2}}_{\ol{X}_{1,t}}},
\end{align*}
which completes the proof of \eqref{equ:Linf,g1,g2,0}--\eqref{equ:Linf,g1,g2,1}.

For \eqref{equ:Linf,g2,L2,0}--\eqref{equ:Linf,g2,L2,1}, analogous to \eqref{equ:items,I2,i}, we have
\begin{align*}
e^{\la_{0}t}|w_{l} \pa_{x}^{k} g_{2}(t)|\leq e^{\la_{0}t}\sum_{i=1}^{3}\abs{I_{2,i}},
\end{align*}
where
\begin{equation*}
\left\{
\begin{aligned}
I_{2,1}= & e^{-\int_0^t \mathcal{A}(\tau, V(\tau)) d  \tau}\left(w_{l} \pa_{x}^{k}g_{2}\right)(0,X(0), V(0)), \\
I_{2,2}=& \int_0^t e^{-\int_s^t \mathcal{A}(\tau, V(\tau)) d  \tau}\lrc{\chi_{M} \mu^{-\frac{1}{2}} w_{l} K_{\mu}\left(\pa_{x}^{k}g_{1}\right)}(s,X(s),V(s)) d  s,\\
I_{2,3}=&\int_0^t e^{-\int_s^t \mathcal{A}(\tau, V(\tau)) d  \tau}\left[w_{l} K\left(\pa_{x}^{k}g_{2}\right)\right](s,X(s),V(s)) d  s.
\end{aligned}
\right.
\end{equation*}

For $I_{2,1}$ and $I_{2,2}$, by Lemma \ref{lemma:Q,bilinear estimate,Lp}, we derive
\begin{align*}
&e^{\la_{0}t}|I_{2,1}|+e^{\la_{0}t}|I_{2,2}|\\
\leq& C_{M}
\lrs{\int_{0}^{t}e^{-\frac{\nu_{0}}{2}(t-s)}e^{\la_{0}(t-s)}ds} \lrs{\sup_{s\in[0,t]}e^{\la_{0}s}\n{K_{\mu} (\pa_{x}^{k}g_{1})}_{L_{x,v}^{\infty}}}
\leq C_{M}\n{g_{1}}_{\ol{X}_{k,t}}.
\end{align*}

For $I_{2,3}$, we decompose it as
\begin{align*}
I_{2,3}=J_{1}+J_{2},
\end{align*}
where
\begin{align*}
J_{1}=&\int_{t-\ve}^{t} e^{-\int_s^t \mathcal{A}(\tau, V(\tau)) d  \tau}\left[w_{l} K\left(\pa_{x}^{k}g_{2}\right)\right](s,X(s),V(s)) d  s,\\
J_{2}=&\int_0^{t-\ve} e^{-\int_s^t \mathcal{A}(\tau, V(\tau)) d  \tau}\left[w_{l} K\left(\pa_{x}^{k}g_{2}\right)\right](s,X(s),V(s)) d  s.
\end{align*}

For $J_{1}$, using Lemma \ref{lemma:Gamma,bilinear estimate}, we obtain
\begin{align*}
e^{\la_{0}t}|J_{1}|\leq \ve \sup_{s\in[0,t]}e^{\la_{0}s}\n{w_{l}K(\pa_{x}^{k}g_{2})(s)}_{L_{x,v}^{\infty}}\leq \ve
\n{g_{2}}_{\ol{X}_{k,t}}.
\end{align*}

For $J_{2}$, we proceed with the following decomposition
\begin{align*}
e^{\la_{0}t}|J_{2}|=&e^{\la_{0}t}\int_{0}^{t-\ve} e^{-\int_s^t \mathcal{A}(\tau, V(\tau)) d  \tau}w_{l}(V(s)) k(V(s),\xi)\left(\pa_{x}^{k}g_{2}\right)(s,X(s),\xi) d\xi d  s\\
\leq& e^{\la_{0}t}J_{2,1}+e^{\la_{0}t}J_{2,2}+e^{\la_{0}t}J_{2,3}+e^{\la_{0}t}J_{2,4},
\end{align*}
where $k(v,v_{*})$ denotes the kernel function of the collision operator $K$ and the terms $J_{2,i}$ are defined as
\begin{equation*}
\left\{
\begin{aligned}
J_{2,1}=&\int_{0}^{t-\ve} e^{-\frac{\nu_{0}}{2}(t-s)}w_{l}(V(s)) \babs{k_{1,M}(V(s),\xi)\left(\pa_{x}^{k}g_{2}\right)(s,X(s),\xi)} d\xi d  s,\\
J_{2,2}=&\int_{0}^{t-\ve} e^{-\frac{\nu_{0}}{2}(t-s)}w_{l}(V(s)) \babs{k_{2,M}(V(s),\xi)\left(\pa_{x}^{k}g_{2}\right)(s,X(s),\xi)} d\xi d  s,\\
J_{2,3}=&\int_{0}^{t-\ve} e^{-\frac{\nu_{0}}{2}(t-s)}w_{l}(V(s)) \babs{k_{3,M}(V(s),\xi)\left(\pa_{x}^{k}g_{2}\right)(s,X(s),\xi)} d\xi d  s,\\
J_{2,4}=&\int_{0}^{t-\ve} e^{-\frac{\nu_{0}}{2}(t-s)}w_{l}(V(s)) \babs{k_{4,M}(V(s),\xi)\left(\pa_{x}^{k}g_{2}\right)(s,X(s),\xi)} d\xi d  s,
\end{aligned}
\right.
\end{equation*}
with the kernel functions given by
\begin{equation*}
\left\{
\begin{aligned}
k_{1,M}(V(s),\xi)=&1_{\lr{|V(s)|\geq M}}k(V(s),\xi),\\
k_{2,M}(V(s),\xi)=&1_{\lr{|V(s)|\leq M,\ |\xi|\geq 2M}}k(V(s),\xi),\\
k_{3,M}(V(s),\xi)=&1_{\lr{|V(s)|\leq M,\ |\xi|\geq 2M,\ |V(s)-\xi|\leq \frac{1}{M}}}k(V(s),\xi),\\
k_{4,M}(V(s),\xi)=&1_{\lr{|V(s)|\leq M,\ |\xi|\leq 2M,\ |V(s)-\xi|\geq \frac{1}{M}}}k(V(s),\xi).\\
\end{aligned}
\right.
\end{equation*}

For $J_{2,1}$, due to that $|V(s)|\geq M$, we apply \eqref{equ:kernel estimate,K,L1} in Lemma \ref{Lemma,L,K,estimates} to derive
\begin{align}\label{equ:estimate,J21,kernel}
\int_{\R^{3}} w_{l}(V(s))k_{1,M}(V(s),\xi)w_{-l}(\xi)d\xi\lesssim
 \frac{1_{\lr{|V(s)|\geq M}}}{1+|V(s)|}\lesssim \frac{1}{M}.
\end{align}
Consequently, we get
\begin{align*}
e^{\la_{0}t}J_{2,1}\leq& \sup_{s\in[0,t]}e^{\la_{0}s}\n{w_{l}\pa_{x}^{k}g_{2}(s)}_{L_{x,v}^{\infty}}
\lrs{\int_{0}^{t-\ve}e^{-\frac{\nu_{0}}{2}(t-s)}e^{\la_{0}(t-s)}ds}\\
&\times \sup_{s\in[0,t]}\bbabs{\int_{\R^{3}} w_{l}(V(s))k_{1,M}(V(s),\xi)w_{-l}(\xi)d\xi}
\lesssim \frac{1}{M}\n{g_{2}}_{\ol{X}_{k,t}}.
\end{align*}

For $J_{2,2}$, the conditions $|V(s)|\leq M$ and $|\xi|\geq 2M$ imply that $|V(s)-\xi|\geq M$. Thus, using again \eqref{equ:kernel estimate,K,L1} in Lemma \ref{Lemma,L,K,estimates}, we obtain
\begin{align*}
&\bbabs{\int_{\R^{3}}1_{\lr{|V(s)|\leq M,\ |\xi|\geq 2M}} w_{l}(V(s))k(V(s),\xi)w_{-l}(\xi)d\xi}\\
\lesssim& e^{\frac{-\delta M^{2}}{8}}\bbabs{\int_{\R^{3}}1_{\lr{|V(s)-\xi|\geq M}} w_{l}(V(s))k(V(s),\xi)w_{-l}(\xi)e^{\frac{\delta|V(s)-\xi|^{2}}{8}}d\xi}\lesssim e^{\frac{-\delta M^{2}}{8}},
\end{align*}
which yields that
\begin{align*}
e^{\la_{0}t}J_{2,2}\lesssim e^{\frac{-\delta M^{2}}{8}}\n{g_{2}}_{\ol{X}_{k,t}}.
\end{align*}

For $J_{2,3}$, by the pointwise estimate \eqref{equ:kernel estimate,K,pointwise} in Lemma \ref{Lemma,L,K,estimates}, we get
\begin{align*}
&\int_{\R^{3}} 1_{\lr{|V(s)|\leq M,\ |\xi|\geq 2M,\ |V(s)-\xi|\leq \frac{1}{M}}}
 w_{l}(V(s))k(V(s),\xi)w_{-l}(\xi)d\xi\\
 \leq&\int_{|V(s)-\xi|\leq\frac{1}{M}}|V(s)-\xi|^{-2}d\xi\lesssim \frac{1}{M}.
\end{align*}

Combining estimates for $I_{2,i}$ and $J_{2,i}$ for $i=1,2,3$, we arrive at
\begin{align}\label{equ:Linf,g2,L2,1,first,proof}
&e^{\la_{0}t}|w_{l} \pa_{x}^{k} g_{2}(t)|\\
\leq &C_{M}\n{g_{1}}_{\ol{X}_{k,t}}+(\ve+M^{-1})
\n{g_{2}}_{\ol{X}_{k,t}}+e^{\la_{0}t}J_{2,4}\notag\\
= & C_{M}\n{g_{1}}_{\ol{X}_{k,t}}+(\ve+M^{-1})
\n{g_{2}}_{\ol{X}_{k,t}}\notag\\
&+\int_{0}^{t-\ve} e^{-\frac{\nu_{0}}{2}(t-s)}e^{\la_{0}(t-s)}w_{l}(V(s)) \babs{k_{4,M}(V(s),\xi)\left(e^{\la_{0}s}\pa_{x}^{k}g_{2}\right)(s,X(s),\xi)} d\xi d  s.\notag
\end{align}
Substituting this back into \eqref{equ:Linf,g2,L2,1,first,proof}, we refine the estimate to
\begin{align}\label{equ:Linf,L2,J12}
e^{\la_{0}t}|w_{l} \pa_{x}^{k} g_{2}(t)|
\leq & C_{M}\n{g_{1}}_{\ol{X}_{k,t}}+(\ve+M^{-1})
\n{g_{2}}_{\ol{X}_{k,t}}+e^{\la_{0}t}J_{2,4,1}+e^{\la_{0}t}J_{2,4,2},
\end{align}
where
\begin{align*}
e^{\la_{0}t}J_{2,4,1}=&
\int_{0}^{t-\ve} e^{-\frac{\nu_{0}}{2}(t-s)}e^{\la_{0}(t-s)}\int_{\R^{3}}w_{l}(V(s))\babs{k_{4,M}(V(s),\xi)}w_{-l}(\xi) d\xi d  s\\
&\times\lrs{C_{M}\n{g_{1}}_{\ol{X}_{k,t}}+(\ve+M^{-1})
\n{g_{2}}_{\ol{X}_{k,t}}},
\end{align*}
and
\begin{align*}
&e^{\la_{0}t}J_{2,4,2}\\
=&
\int_{0}^{t-\ve} e^{-\frac{\nu_{0}}{2}(t-s)}e^{\la_{0} (t-s)}\int_{\R^{3}}w_{l}(V(s))\babs{k_{4,M}(V(s),\xi)}w_{-l}(\xi) \\
&\int_{0}^{s-\ve}e^{-\frac{\nu_{0}}{2}(s-\wt{s})}e^{\la_{0}(s-\wt{s})}\int_{\R^{3}}
w_{l}(\wt{V}(\wt{s}))\babs{k_{4,M}(\wt{V}(\wt{s}),\wt{\xi})}\lrs{e^{\la_{0}\wt{s}}\pa_{x}^{k}g_{2}}(\wt{s},\wt{X}(\wt{s}),\wt{\xi}) d\wt{\xi}d\wt{s}d\xi d  s
\end{align*}
with
\begin{align*}
\wt{V}(\wt{s})=V(\wt{s};s,X(s),\xi),\quad \wt{X}(\wt{s})=X(\wt{s};s,X(s),\xi).
\end{align*}

For $e^{\la_{0}t}J_{2,4,1}$, following the estimate \eqref{equ:estimate,J21,kernel} for $e^{\la_{0}t}J_{2,1}$, we have
\begin{align}\label{equ:Linf,L2,J241}
e^{\la_{0}t}J_{2,4,1}
\leq&\int_{0}^{t-\ve} e^{-\frac{\nu_{0}}{2}(t-s)}e^{\la_{0}(t-s)} d  s
\lrs{C_{M}\n{g_{1}}_{\ol{X}_{k,t}}+(\ve+M^{-1})
\n{g_{2}}_{\ol{X}_{k,t}}}\\
\leq&C_{M}\n{g_{1}}_{\ol{X}_{k,t}}+(\ve+M^{-1})
\n{g_{2}}_{\ol{X}_{k,t}}.\notag
\end{align}

For $e^{\la_{0}t}J_{2,4,2}$, using the pointwise estimate \eqref{equ:kernel estimate,K,pointwise} and the cutoff property, we get
\begin{align*}
w_{l}(V(s))\babs{k_{4,M}(V(s),\xi)}\leq C_{M}1_{\lr{|V(s)|\leq M,\ |\xi|\leq 2M,\ |V(s)-\xi|\geq \frac{1}{M}}},
\end{align*}
and hence obtain
\begin{align*}
e^{\la_{0}t}J_{2,4,2}\lesssim& \int_{0}^{t-\ve} e^{-\frac{\nu_{0}}{2}(t-s)}e^{\la_{0} (t-s)}
\int_{0}^{s-\ve}e^{-\frac{\nu_{0}}{2}(s-\wt{s})}e^{\la_{0}(s-\wt{s})}\\
&\int_{|\xi|\leq 2M}\int_{|\wt{\xi}|\leq 2M}
\babs{\lrs{e^{\la_{0}\wt{s}}\pa_{x}^{k}g_{2}}(\wt{s},\wt{X}(\wt{s}),\wt{\xi})} d\wt{\xi}d\wt{s}d\xi d  s.
\end{align*}
Recalling the characteristic trajectory given in \eqref{equ:characteristic ODE}--\eqref{equ:characteristic ODE,solution} that
\begin{equation*}
\left\{\begin{aligned}
X_{1}(s ; t, x, v) & =e^{\beta t}\left(e^{-\beta t}x_{1}-(t-s) v_{1}-\frac{1}{2} \alpha(t-s)^2 v_{2}\right), \\
X_{i}(s ; t, x, v) & =e^{\beta t}\left(e^{-\beta t}x_{i}-(t-s) v_i\right),\quad i=2,3, \\
V_1(s ; t, x, v) & =e^{\beta(t-s)}\left(v_{1}+\alpha v_{2}(t-s)\right), \\
V_i(s ; t, x, v) & =e^{\beta(t-s)} v_i ,\quad i=2,3,
\end{aligned}\right.
\end{equation*}
we obtain an explicit formula that
\begin{align*}
&\wt{X}(\wt{s})=X(\wt{s};s,X(s),\xi)\\
=&\begin{pmatrix}
e^{\beta s}\left(e^{-\beta s}X_{1}(s)-(s-\wt{s}) \xi_1-\frac{1}{2}
\alpha(s-\wt{s})^2 \xi_2\right)\\
e^{\beta s}(e^{-\beta s}X_{2}(s)-(s-\wt{s})\xi_{2})\\
e^{\beta s}(e^{-\beta s}X_{3}(s)-(s-\wt{s})\xi_{3})\\
\end{pmatrix}\\
=&\begin{pmatrix}
x_{1}-e^{\beta t}(t-s) v_{1}-\frac{1}{2} \alpha e^{\beta t}(t-s)^2 v_{2}-e^{\beta s}(s-\wt{s}) \xi_1-\frac{1}{2}
\alpha e^{\beta s}(s-\wt{s})^2 \xi_2\\
x_{2}-e^{\beta t}(t-s) v_{2}-e^{\beta s}(s-\wt{s})\xi_{2}\\
x_{3}-e^{\beta t}(t-s) v_{3}-e^{\beta s}(s-\wt{s})\xi_{3}\\
\end{pmatrix}.
\end{align*}
Equivalently,
\begin{align*}
\wt{X}(\wt{s})=\eta-\mathcal{T}\xi,
\end{align*}
where
\begin{align*}
\eta=\begin{pmatrix}
x_{1}-e^{\beta t}(t-s)v_{1}-\frac{1}{2}\al e^{\beta t}(t-s)^{2}v_{2}\\
x_{2}-e^{\beta t}(t-s)v_{2}\\
x_{3}-e^{\beta t}(t-s)v_{3}
\end{pmatrix},\quad
\xi=\begin{pmatrix}
\xi_{1}\\
\xi_{2}\\
\xi_{3}
\end{pmatrix},
\end{align*}
and
\begin{align*}
\mathcal{T}=\begin{pmatrix}
e^{\beta s}(s-\wt{s}) &\frac{1}{2}\al e^{\beta s}(s-\wt{s})^{2} &0 \\
0& e^{\beta s}(s-\wt{s}) &0\\
0& 0& e^{\beta s}(s-\wt{s})
\end{pmatrix}.
\end{align*}
Performing the change of variables $y=\wt{X}(\wt{s})=\eta-\mathcal{T}\xi$, we derive
\begin{align*}
&\int_{|\xi|\leq 2M}\int_{|\wt{\xi}|\leq 2M}
\babs{\lrs{e^{\la_{0}\wt{s}}\pa_{x}^{k}g_{2}}(\wt{s},\wt{X}(s),\wt{\xi})} d\wt{\xi}d\xi \\
=&\int_{|\xi|\leq 2M}\int_{|\wt{\xi}|\leq 2M}
\babs{\lrs{e^{\la_{0}\wt{s}}\pa_{x}^{k}g_{2}}(\wt{s},\eta-\mathcal{T}\xi,\wt{\xi})} d\wt{\xi}d\xi \\
\leq&\int_{|y-\eta|\lesssim M\n{\mathcal{T}}_{L^{1}}}\int_{|\wt{\xi}|\leq 2M}
\babs{\lrs{e^{\la_{0}\wt{s}}\pa_{x}^{k}g_{2}}(\wt{s},y,\wt{\xi})} |\operatorname{det}\mathcal{T}|^{-1}d\wt{\xi}dy \\
\lesssim&(\n{\mathcal{T}}_{L^{1}})^{3}|\operatorname{det}\mathcal{T}|^{-1} \lrs{\int_{\T^{3}}\int_{\R^{3}} \babs{\lrs{e^{\la_{0}\wt{s}}\pa_{x}^{k}g_{2}}(\wt{s},y,\xi)}^{2}dyd\xi}^{\frac{1}{2}}\\
\lesssim& \lrs{1+(s-\wt{s})^{3}}\sup_{\wt{s}\in[0,t]}e^{\la_{0}\wt{s}}\n{\pa_{x}^{k}g_{2}(\wt{s})}_{L_{x,v}^{2}},
\end{align*}
where the last inequality follows from the estimates
\begin{align*}
\n{\mathcal{T}}_{L^{1}}\lesssim e^{\beta s}(s-\wt{s})+\al e^{\beta s}(s-\wt{s})^{2},\quad
|\operatorname{det}\mathcal{T}|^{-1}=e^{-3\beta s}(s-\wt{s})^{-3}.
\end{align*}
Consequently, we obtain
\begin{align}\label{equ:Linf,L2,J242}
e^{\la_{0}t}J_{2,4,2}
\leq&\sup_{s\in[0,t]}\n{e^{\la_{0}s}\pa_{x}^{k}g_{2}(s)}_{L_{x,v}^{2}} \int_{0}^{t-\ve} e^{-\frac{\nu_{0}}{2}(t-s)}e^{\la_{0} (t-s)}\\
&\int_{0}^{s-\ve}e^{-\frac{\nu_{0}}{2}(s-\wt{s})}e^{\la_{0}(s-\wt{s})}\lrs{1+(s-\wt{s})^{3}}dsd\wt{s}\notag\\
\lesssim& \sup_{s\in[0,t]}e^{\la_{0}s}\n{\pa_{x}^{k}g_{2}(s)}_{L_{x,v}^{2}}.\notag
\end{align}

Combining estimates \eqref{equ:Linf,L2,J12}, \eqref{equ:Linf,L2,J241}, and \eqref{equ:Linf,L2,J242}, we arrive at
\begin{align}\label{equ:g2,Linf,L2,proof}
\n{g_{2}}_{\ol{X}_{k,t}}\lesssim \n{g_{1}}_{\ol{X}_{k,t}}+
\sup_{s\in[0,t]}e^{\la_{0}s}\n{\pa_{x}^{k}g_{2}(s)}_{L_{x,v}^{2}}.
\end{align}

Repeating the above argument employed in \eqref{equ:g2,Linf,L2,proof} for the case $k=0$, $\la_{0}=0$, we similarly obtain
\begin{align*}
\n{g_{2}}_{\ol{X}_{0,t}}\lesssim \n{g_{1}}_{\ol{X}_{0,t}}+
\sup_{s\in[0,t]}\n{g_{2}(s)}_{L_{x,v}^{2}}.
\end{align*}
Hence, we complete the proof of \eqref{equ:Linf,g2,L2,0}--\eqref{equ:Linf,g2,L2,1}. The proof of Lemma \ref{lemma,Linf,estimate} is done.
\end{proof}

\section{$L_{x,v}^{2}$ Energy Estimates}\label{sec:L2 Energy Estimates}
To provide a closed estimate for the weighted $L_{x,v}^{\infty}$ estimates derived in Section \ref{sec:Weighted Estimates}, we establish $L_{x,v}^{2}$ energy estimates through a systematic approach combining macro-micro decomposition and low-high frequency analysis. The macro-micro decomposition separates the distribution function
\begin{align*}
\wt{g}=\mathbf{P}_{0}\wt{g}+\mathbf{P}_{1}\wt{g},
\end{align*}
into its macroscopic (fluid) part $\mathbf{P}_{0}\wt{g}$ and microscopic (kinetic) part $\mathbf{P}_{1}\wt{g}$.

In Section \ref{sec:Lower Energy Estimates}, we conduct a detailed analysis of $L_{x,v}^{2}$ low energy estimates for the microscopic component, establishing bounds for the kinetic part of the solution. In Section \ref{sec:Moment Estimates}, we focus on estimating the macroscopic component, with attention to the second-order moment of zero-frequency modes, which captures essential macroscopic dynamics in non-conservative systems. Finally, in Section \ref{sec:Higher Energy Estimates}, we provide the higher energy estimate that upgrade the bounds obtained in Sections \ref{sec:Lower Energy Estimates}--\ref{sec:Moment Estimates}
to the full regularity level required for the nonlinear analysis.

\subsection{Lower Energy Estimates}\label{sec:Lower Energy Estimates}
We begin by defining the macroscopic and microscopic components of $\wt{g}$
\begin{align*}
\mathbf{P}_0 \wt{g}=:&\left\{a+\mathbf{b} \cdot v+c\frac{|v|^2-3}{\sqrt{6}}\right\} \sqrt{\mu},\quad  \mathbf{P}_{1} \wt{g}=:\wt{g}-\mathbf{P}_0 \wt{g},
\end{align*}
with the coefficients given by
\begin{equation*}
\left\{
\begin{aligned}
a(t,x)=&\lra{\wt{g},\sqrt{\mu}}=\lra{\wt{f},1},\\
\mathbf{b}(t,x)=&\lra{\wt{g},v\sqrt{\mu}}=\lra{\wt{f},v},\\
c(t,x)=&\lra{\wt{g},\frac{|v|^{2}-3}{\sqrt{6}}\sqrt{\mu}}=\lra{\wt{f},\frac{|v|^{2}-3}{\sqrt{6}}}.
\end{aligned}
\right.
\end{equation*}
Here, $\mathbf{P}_{0}$ is also
the projection operator onto the kernel of the linearized collision operator $L$ defined by \eqref{equ:notation,Lg}.

We now present the low energy estimates for both macroscopic and microscopic components.
\begin{lemma}\label{lemma:g2,L2,estimate,p0,p1}
Let $[g_{1},g_{2}]$ be the solution to \eqref{equ:g1,equation}--\eqref{equ:g2,equation}. The following energy estimates hold.
\begin{enumerate}[$(1)$]
\item The macroscopic estimate:
\begin{align}\label{equ:macro estimate,lower energy}
\sup_{s\in[0,t]} \left\|\mathbf{P}_0 g_{2}(s)\right\|_{L_{x,v}^{2}}\lesssim
\sup_{s\in[0,t]}\n{c(s)}_{L_{x}^{2}}+
\sup_{s\in[0,t]}\n{\pa_{x}[a,\mathbf{b}](s)}_{L_{x}^{2}}+
\sup_{s\in[0,t]}\n{w_{l}g_{1}(s)}_{L_{x,v}^{\infty}}.
\end{align}
\item
The microscopic estimate:
\begin{align}\label{equ:micro estimate,lower energy}
 \sup_{s\in[0,t]} \n{\mathbf{P}_{1} g_{2}(s)}_{L_{x,v}^{2}}
\lesssim & \sup_{s\in[0,t]} \n{c(s)}_{L_{x}^{2}}+\sup_{s\in[0,t]} \n{\pa_{x}[a,\mathbf{b}](s)}_{L_{x}^{2}}  +\sup_{s\in[0,t]} \left\|w_{l} g_{1}(s)\right\|_{L_{x,v}^{\infty}} .
\end{align}
\end{enumerate}
In particular, we have
\begin{align}\label{equ:g2,L2,estimate}
\sup_{s\in[0,t]}\n{g_{2}(s)}_{L_{x,v}^{2}}\lesssim&
  \sup_{s\in[0,t]} \n{c(s)}_{L_{x}^{2}}
   +
\sum_{k=0,1} \n{g_{1}}_{\ol{X}_{k,t}}+\sup_{s\in[0,t]} e^{\la_{0}s}\n{\pa_{x}g_{2}}_{L_{x}^{2}}.
\end{align}

\end{lemma}
\begin{proof}
For the macroscopic estimate \eqref{equ:macro estimate,lower energy}, noting that $\wt{g}=\mu^{-\frac{1}{2}}g_{1}+g_{2}$, we get
\begin{align*}
\n{\mathbf{P}_0 g_{2}}_{L_{v}^{2}}\leq& \n{\mathbf{P}_0 \wt{g}}_{L_{v}^{2}}+\n{\mathbf{P}_0 (\mu^{-\frac{1}{2}}g_{1})}_{L_{v}^{2}}\lesssim |[a,\mathbf{b},c]|+\bbabs{\bblra{g_{1},[1,v,\frac{|v|^{2}-3}{\sqrt{6}}]}}\\
\lesssim& |[a,\mathbf{b},c]|+\n{w_{l}g_{1}}_{L_{v}^{\infty}}.\notag
\end{align*}
By the mass and momentum conservation laws \eqref{equ:mass,moment,conservation} in Lemma \ref{lemma:evolution equation,moments}, and the Poincar\'e inequality that $\n{P_{\neq 0}^{x}u}_{L_{x}^{2}}\leq \n{u}_{L_{x}^{2}}$, we obtain
\begin{align*}
\n{\mathbf{P}_0 g_{2}}_{L_{x,v}^{2}}\lesssim \n{c}_{L_{x}^{2}}+\n{\pa_{x}[a,\mathbf{b}]}_{L_{x}^{2}}+\n{w_{l}g_{1}}_{L_{x,v}^{\infty}},
\end{align*}
which completes the proof of \eqref{equ:macro estimate,lower energy}.

For the microscopic part estimate \eqref{equ:micro estimate,lower energy},
we recall that
\begin{equation}\label{equ:g2,equation,energy estimate,high}
\left\{
\begin{aligned}
&\pa_t g_{2}+e^{\beta t} \sum_{i=1}^{3}v_{i}\pa_{x_{i}} g_{2}-\beta \nabla_{v} \cdot\left(v g_{2}\right)-\alpha v_{2} \pa_{v_{1}} g_{2}+L g_{2} =\mu^{-1 / 2}\chi_{M}K_{\mu} g_{1},  \\
&g_{2}(0, x, v)=0 .
\end{aligned}\right.
\end{equation}
Testing \eqref{equ:g2,equation,energy estimate,high} by $g_{2}$, we obtain
\begin{align*}
\lra{\pa_t g_{2}+e^{\beta t} \sum_{i=1}^{3}v_{i}\pa_{x_{i}} g_{2}-\beta \nabla_{v} \cdot\left(v g_{2}\right)-\alpha v_{2} \pa_{v_{1}} g_{2}+L g_{2} ,g_{2}}=\lra{\mu^{-1 / 2}\chi_{M}K_{\mu} g_{1},g_{2}},
\end{align*}
which, together with integration by parts, yields
\begin{align}\label{equ:g2,equation,test,g2}
\frac{1}{2}\pa_{t}\lra{g_{2},g_{2}}-\frac{3\beta}{2}\lra{g_{2},g_{2}}+\lra{Lg_{2},g_{2}}=
\lra{\mu^{-1 / 2}\chi_{M}K_{\mu} g_{1},g_{2}}.
\end{align}
By Lemma \ref{Lemma,L,K,estimates}, we have
\begin{align}\label{equ:L,lower bound,proof}
\lra{Lg_{2},g_{2}}\geq \delta_{0}\lra{\mathbf{P}_{1}g_{2},\mathbf{P}_{1}g_{2}}.
\end{align} Using Young's inequality and estimate \eqref{equ:Kmu,estimate,Lp} in Lemma \ref{lemma:Q,bilinear estimate,Lp}, we derive
\begin{align}\label{equ:L2,energy estimate,force,proof}
\lra{\mu^{-1 / 2}\chi_{M} K_{\mu} g_{1},g_{2}}\leq &
\ve \n{g_{2}}_{L_{x,v}^{2}}^{2}+C_{\ve}\n{\mu^{-1 / 2}\chi_{M} K_{\mu} g_{1}}_{L_{x,v}^{2}}^{2}\\
\leq&
\ve\n{g_{2}}_{L_{x,v}^{2}}^{2}+C_{\ve}C_{M}\n{ K_{\mu} g_{1}}_{L_{x,v}^{2}}^{2}\notag\\
\leq& \ve\n{g_{2}}_{L_{x,v}^{2}}^{2}+C_{\ve,M}\n{w_{l}g_{1}}_{L_{x,v}^{\infty}}^{2}.\notag
\end{align}
Putting \eqref{equ:L,lower bound,proof} and \eqref{equ:L2,energy estimate,force,proof} into \eqref{equ:g2,equation,test,g2}, and choosing that $\la_{0}\leq \delta_{0}-\frac{3\beta}{2}-\ve$, we arrive at
\begin{align}\label{equ:micro estimate,lower energy,proof}
\frac{d}{dt}\n{\mathbf{P}_{1}g_{2}}_{L_{x,v}^{2}}^{2}+\frac{d}{dt}\n{\mathbf{P}_{0}
g_{2}}_{L_{x,v}^{2}}^{2}+\la_{0}\n{\mathbf{P}_{1}g_{2}}_{L_{x,v}^{2}}^{2}\leq \ve\n{\mathbf{P}_{0}g_{2}}_{L_{x,v}^{2}}^{2}
+C_{\ve,M}\n{w_{l}g_{1}}_{L_{x,v}^{\infty}}^{2}.
\end{align}
Then by Gronwall's inequality, we get
\begin{align*}
\n{\mathbf{P}_{1}g_{2}(t)}_{L_{x,v}^{2}}^{2}\lesssim& e^{-\la_{0}t}\n{\mathbf{P}_{1}g_{2}(0)}_{L_{x,v}^{2}}^{2}+
\bbabs{\int_{0}^{t}
e^{-\la_{0}(t-s)}\frac{d}{ds}\n{\mathbf{P}_{0}g_{2}}_{L_{x,v}^{2}}^{2}ds}\\
&+\int_{0}^{t}
e^{-\la_{0}(t-s)}\lrc{\n{\mathbf{P}_{0}g_{2}(s)}_{L_{x,v}^{2}}^{2}+\n{w_{l}g_{1}(s)}_{L_{x,v}^{\infty}}^{2}}ds\\
\lesssim& e^{-\la_{0}t}\n{\mathbf{P}_{1}g_{2}(0)}_{L_{x,v}^{2}}^{2}+\sup_{s\in[0,t]}\n{\mathbf{P}_{0}g_{2}(s)}_{L_{x,v}^{2}}^{2}+
\sup_{s\in[0,t]}\n{w_{l}g_{1}(s)}_{L_{x,v}^{\infty}}^{2},
\end{align*}
which, together with the zero initial data and estimate \eqref{equ:macro estimate,lower energy} on $\n{\mathbf{P}_{0}g_{2}}_{L_{x,v}^{2}}$, completes the proof of estimate \eqref{equ:micro estimate,lower energy}.

Combining \eqref{equ:macro estimate,lower energy} and \eqref{equ:micro estimate,lower energy}, we obtain
\begin{align*}
\sup_{s\in[0,t]}\n{g_{2}(s)}_{L_{x,v}^{2}}
\lesssim& \sup_{s\in[0,t]} \n{c(s)}_{L_{x}^{2}}+
\sup_{s\in[0,t]} \n{\pa_{x}[a,\mathbf{b}](s)}_{L_{x}^{2}}+\sup_{s\in[0,t]} \left\|w_{l} g_{1}(s)\right\|_{L_{x,v}^{\infty}} \\
\lesssim&  \sup_{s\in[0,t]} \n{c(s)}_{L_{x}^{2}}
   +\sup_{s\in[0,t]} e^{\la_{0}s}\n{\pa_{x}g_{2}}_{L_{x}^{2}}+
\sum_{k=0,1} \n{g_{1}}_{\ol{X}_{k,t}},
\end{align*}
where in the last inequality we have used that
\begin{align*}
\n{\pa_{x}[a,\mathbf{b}]}_{L_{x}^{2}}=&\n{\lra{\pa_{x}\wt{f},[1,v]}}_{L_{x}^{2}}=
\n{\lra{\pa_{x}g_{1}+\pa_{x}g_{2}\mu^{\frac{1}{2}},[1,v]}}_{L_{x}^{2}}\\
\lesssim& \n{w_{l}\pa_{x}g_{1}}_{L_{x,v}^{\infty}}+\n{\pa_{x}g_{2}}_{L_{x,v}^{2}}.
\end{align*}
Hence, we have completed the proof of \eqref{equ:g2,L2,estimate}.
\end{proof}

\subsection{Moment Estimates}\label{sec:Moment Estimates}
To close the lower energy estimate via Lemma \ref{lemma:g2,L2,estimate,p0,p1}, a global estimate for the energy moment $\sup_{s\in[0,t]}\n{c(s)}_{L_{x}^{2}}$ is required. As a first step, we decompose the solution into zero-frequency and nonzero-frequency modes:
\begin{align*}
\sup_{s\in[0,t]}\n{c(s)}_{L_{x}^{2}}\leq \sup_{s\in[0,t]}\n{P_{0}^{x}c(s)}_{L_{x}^{2}}+
\sup_{s\in[0,t]}\n{P_{\neq 0}^{x}c(s)}_{L_{x}^{2}},
\end{align*}
where the projectors are defined by
 \begin{align*}
  P_{0}^{x}\phi=:\int_{\T^{3}}\phi(x)dx,\quad P_{\neq 0}^{x}\phi=:\phi-P_{0}^{x}\phi.
 \end{align*}

In this section, we focus on estimating the zero-frequency mode.
We begin by introducing the following hydrodynamic moments
\begin{equation}\label{equ:hydrodynamic moments}
\left\{
\begin{aligned}
E(t,x)=&\lra{\wt{g},|v|^{2}\sqrt{\mu}}=\lra{\wt{f},|v|^{2}},\\
d_{ij}(t,x)=&\lra{\wt{g},v_{i}v_{j}\sqrt{\mu}}=\lra{\wt{f},v_{i}v_{j}},\quad
d_{ijk}(t,x)=\lra{\wt{g},v_{i}v_{j}v_{k}\sqrt{\mu}}=\lra{\wt{f},v_{i}v_{j}v_{k}},\\
B_{i}=&\lra{\wt{g},v_{i}(|v|^{2}-5)\sqrt{\mu}}=\lra{\wt{f},v_{i}(|v|^{2}-5)},\\
B_{ij}=&\lra{\wt{g},v_{i}v_{j}(|v|^{2}-5)\sqrt{\mu}}=\lra{\wt{f},v_{i}v_{j}(|v|^{2}-5)}.\\
\end{aligned}
\right.
\end{equation}

As a preliminary step, we derive the evolution equations for these moments.
\begin{lemma}\label{lemma:evolution equation,moments}
Let $\wt{\mathcal{Q}}=Q(\wt{f},\wt{f})+Q_{sym}(\wt{f},G)$. Then there holds that
\begin{align}
&\pa_{t}a+e^{\beta t}\sum_{i=1}^{3}\pa_{x_{i}}b_{i}=0,\label{equmomentequationa}\\
&\pa_t b_1+\beta b_1+ \alpha b_2+e^{\beta t} \sum_{i=1}^{3}\pa_{x_{i}}d_{1i}=0,\label{equmomentequationb1}\\
&\pa_t b_i+\beta b_i+e^{\beta t} \sum_{j=1}^{3}\pa_{x_{j}} d_{ij}=0, \quad i=2,3,\label{equmomentequationbi}\\
&\pa_{t}c+\beta (\sqrt{6}a+2c)+\al  \frac{\sqrt{6}}{3}d_{12}+e^{\beta t}\sum_{i=1}^{3}\pa_{x_{i}}\lra{\wt{f},v_{i}\frac{|v|^{2}-3}{\sqrt{6}}}=0,\label{equ:moment,equation,c}\\
&\pa_{t}E+2\beta E+2\al d_{12}+e^{\beta t}\sum_{i=1}^{3}\pa_{x_{i}}\lra{\wt{f},v_{i}|v|^{2}}=0,\label{equ:moment,equation,E}\\
&\pa_{t}d_{11}+2\beta d_{11}+2\al d_{12}+e^{\beta t}\sum_{i=1}^{3}\pa_{x_{i}}d_{11i}=\lra{\wt{\mathcal{Q}},v_{1}v_{1}},\label{equ:moment,equation,d11}\\
&\pa_{t}d_{22}+2\beta d_{22}+e^{\beta t}\sum_{i=1}^{3}\pa_{x_{i}}d_{22i}=\lra{\wt{\mathcal{Q}},v_{2}v_{2}},\label{equ:moment,equation,d22}\\
&\pa_{t}d_{33}+2\beta d_{33}+e^{\beta t}\sum_{i=1}^{3}\pa_{x_{i}}d_{33i}=\lra{\wt{\mathcal{Q}},v_{3}v_{3}},\label{equ:moment,equation,d33}\\
&\pa_{t}d_{12}+2\beta d_{12}+\al d_{22}+e^{\beta t}\sum_{i=1}^{3}\pa_{x_{i}}d_{12i}=\lra{\wt{\mathcal{Q}},v_{1}v_{2}},\label{equ:moment,equation,d12}\\
&\pa_{t}d_{13}+2\beta d_{13}+\al d_{23}+e^{\beta t}\sum_{i=1}^{3}\pa_{x_{i}}d_{13i}=\lra{\wt{\mathcal{Q}},v_{1}v_{3}},\label{equ:moment,equation,d13}\\
&\pa_{t}d_{23}+2\beta d_{23}+e^{\beta t}\sum_{i=1}^{3}\pa_{x_{i}}d_{23i}=\lra{\wt{\mathcal{Q}},v_{2}v_{3}},\label{equ:moment,equation,d23}\\
&\pa_{t}B_{1}+3\beta B_{1}+10\beta b_{1}+3\al d_{112}+e^{\beta t}\sum_{i=1}^{3} \pa_{x_{i}} B_{1i}=
\lra{\wt{\mathcal{Q}},v_{1}(|v|^{2}-5)},\label{equmomentequationB1}\\
&\pa_{t}B_{2}+3\beta B_{2}+10\beta b_{2}+2\al d_{122}+e^{\beta t}\sum_{i=1}^{3} \pa_{x_{i}} B_{2i}=
\lra{\wt{\mathcal{Q}},v_{2}(|v|^{2}-5)},\label{equ:moment,equation,B2}\\
&\pa_{t}B_{3}+3\beta B_{3}+10\beta b_{3}+2\al d_{132}+e^{\beta t}\sum_{i=1}^{3} \pa_{x_{i}} B_{3i}=
\lra{\wt{\mathcal{Q}},v_{3}(|v|^{2}-5)}.\label{equ:moment,equation,B3}
\end{align}
In particular, we have the mass and moment conservation, that is,
\begin{align}\label{equ:mass,moment,conservation}
P_{0}^{x}a(t)=P_{0}^{x}b(t)=0,\quad t\geq 0.
\end{align}
Moreover, we have a basic estimate on the moments that
\begin{align}\label{equ:moment estimate,basic}
\abs{\lra{\wt{f},p_{n}(v)}}\lesssim \n{w_{l}g_{1}}_{L_{v}^{\infty}}+\n{g_{2}}_{L_{v}^{2}},\quad l> n+3,
\end{align}
where $p_{n}(v)$ is a polynomial of the order $n$.
\end{lemma}
\begin{proof}
The evolution equations %\cref{equmomentequationa,equmomentequationb1,equmomentequationbi}
for the moments are obtained by testing equation \eqref{equ:perturbation equation,wf} through direct computation. For conciseness, we omit the detailed derivation. Regarding the conservation laws, we note that
\begin{equation*}
\left\{
\begin{aligned}
\pa_{t}P_{0}^{x}a=&0,\\
\pa_{t}P_{0}^{x}b_{1}+\al P_{0}^{x}b_{1}+\beta P_{0}^{x}b_{2}=&0,\\
\pa_{t}P_{0}^{x}b_{i}+\beta P_{0}^{x}b_{i}=&0,\quad i=2,3.
\end{aligned}
\right.
\end{equation*}
Combined with the initial condition \eqref{equ:mass,moment,zero,initial} that
$P_{0}^{x}a(0)=P_{0}^{x}\mathbf{b}(0)=0$, this establishes \eqref{equ:mass,moment,conservation}.

The moment estimate \eqref{equ:moment estimate,basic} follows from
\begin{align*}
\abs{\lra{\wt{f},p_{n}(v)}}=\abs{\lra{\mu^{\frac{1}{2}}\wt{g},p_{n}(v)}}=
\abs{\lra{g_{1}+\mu^{\frac{1}{2}}g_{2},p_{n}(v)}}\lesssim
\n{w_{l}g_{1}}_{L_{v}^{\infty}}+\n{g_{2}}_{L_{v}^{2}}.
\end{align*}

We then conclude the proof of Lemma \ref{lemma:evolution equation,moments}.
\end{proof}

Building upon the moment equations established in Lemma \ref{lemma:evolution equation,moments}, we now examine the nonlinear terms $\lra{\wt{\mathcal{Q}},v_{i}v_{j}}$. Our analysis relies on the following fundamental result concerning the collision operator's action on quadratic velocity moments.
\begin{lemma}[{\hspace{-0.01em}\cite[Chapter XII]{TM80},\cite[Proposition 4.10]{JNV19}}]\label{lemma:Tij,Wij}
Let $W_{i j}(v)=v_i v_j$, and define
\begin{align*}
T_{i j}(v,v_{*})=:\frac{1}{2} \int_{\mathbb{S}^2} B_{0}(\cos \theta)\lrc{W_{i, j}(v^{\prime})+W_{i, j}(v_{*}^{\prime})-W_{i, j}(v)-W_{i, j}(v_{*})}d \omega ,
\end{align*}
where the after-collision velocity pair $(v',v_{*}')$ is given by \eqref{equ:pre,post,velocity}.
Then it holds that
\begin{align*}
T_{i j}(v,v_{*})=-b_0\left[\left(v-v_*\right)_i\left(v-v_*\right)_j-\frac{\delta_{i j}}{3}\left|v-v_*\right|^2\right]
\end{align*}
with the constant $b_{0}$ given by \eqref{equ:constant,b0}.
\end{lemma}

Equipped with Lemma \ref{lemma:Tij,Wij}, we derive explicit expressions for the collision operator's action on velocity moments.
\begin{lemma}\label{lemma:cancellation,vij}
We have the following identities:
\begin{align}
 \blra{Q_{sym}(\wt{f},G),v_{i}v_{j}}=&-2b_{0}\lrc{\lra{\wt{f},v_{i}v_{j}}\lra{G,1}
 +\lra{\wt{f},1}\lra{G,v_{i}v_{j}}},\quad i\neq j,\label{equ:QfG,ij}\\
 \lra{Q_{sym}(\wt{f},G),v_{i}v_{i}}
 =&-2b_{0}\lrc{\bblra{\wt{f},v_{i}v_{i}-\frac{|v|^{2}}{3}}
 \lra{G,1}+\lra{\wt{f},1}\bblra{G,v_{i}v_{i}-\frac{|v|^{2}}{3}}},\label{equ:QfG,ii}\\
\lra{Q(\wt{f},\wt{f}),v_{i}v_{j}}=&-2b_{0}\lrc{\lra{\wt{f},v_{i}v_{j}}
\lra{\wt{f},1}-\lra{\wt{f},v_{i}}\lra{\wt{f},v_{j}}},\quad i\neq j,\label{equ:Qff,ij}\\
\lra{Q(\wt{f},\wt{f}),v_{i}v_{i}}=&-2b_{0}\lrc{\bblra{\wt{f},v_{i}v_{i}-\frac{|v|^{2}}{3}}
 \lra{\wt{f},1}+\lra{\wt{f},v_{i}}^{2}-\frac{1}{3}\sum_{j=1}^{3}\lra{\wt{f},v_{j}}^{2}}.\label{equ:Qff,ii}
\end{align}
Moreover, under the conservation laws \eqref{equ:mass,moment,conservation}, the zero-frequency projections satisfy
\begin{align}
P_{0}^{x} \lra{Q_{sym}(\wt{f},G),v_{1}v_{2}}=&-2b_{0}P_{0}^{x} d_{12},\label{equ:QfG,ij,P0}\\
P_{0}^{x} \lra{Q_{sym}(\wt{f},G),v_{i}v_{i}}=&-2b_{0}\lrc{ P_{0}^{x}d_{ii}-\frac{1}{3}P_{0}^{x}E}.\label{equ:QfG,ii,P0}
\end{align}
\end{lemma}
\begin{proof}
By the symmetry and collision-invariant property, we have
\begin{align*}
\lra{Q_{sym}(\psi_{1},\psi_{2}),\phi}
=&\int_{\R^{3}\times \R^{3} \times \mathbb{S}^{2}}B_{0}(\cos \theta) \psi_{1}\psi_{2,*}(\phi'+\phi'_{*}-\phi-\phi_{*}) d \omega d vd v_* .
\end{align*}
Taking $\phi(v)=W_{ij}(v)$ and using Lemma \ref{lemma:Tij,Wij}, we get
\begin{align}\label{equ:QfG,ij,proof}
&\lra{Q_{sym}(\wt{f},G),v_{i}v_{j}}\\
=&\int_{\R^{3}\times \R^{3}}\wt{f}(v)G(v_{*})\lrc{\int_{\mathbb{S}^2} B_{0}(\cos \theta)\lrs{W_{i, j}(v^{\prime})+W_{i, j}(v_{*}^{\prime})-W_{i, j}(v)-W_{i, j}(v_{*})} d\omega} d vd v_* \notag\\
=&2
\int_{\R^{3}\times \R^{3}}\wt{f}(v)G(v_{*})T_{ij}(v,v_{*})dvdv_{*}\notag\\
=&-2b_0\int_{\R^{3}\times \R^{3}}\wt{f}(v)G(v_{*})\left[\left(v-v_*\right)_i\left(v-v_*\right)_j-\frac{\delta_{i j}}{3}\left|v-v_*\right|^2\right]dvdv_{*}.\notag
\end{align}
Expanding the velocity terms
\begin{equation}\label{equ:vij,expansion}
\left\{
\begin{aligned}
(v-v_*)_{i}(v-v_*)_{j}=&v_{i}v_{j}+v_{*,i}v_{*,j}-v_{i}v_{*,j}-v_{j}v_{*,i}, \\
(v-v_{*})_{i}(v-v_{*})_{i}-\frac{|v-v_{*}|^{2}}{3}=&v_{i}v_{i}+v_{*,i}v_{*,i}-2v_{i}v_{*,i}
-\frac{|v|^{2}+|v_{*}|^{2}-2v\cdot v_{*}}{3},
\end{aligned}
\right.
\end{equation}
with $\lra{G,v_{i}}=0$ for $i=1,2,3$, we obtain
\begin{align*}
\lra{Q_{sym}(\wt{f},G),v_{i}v_{j}}=&-2b_{0}\lrc{\lra{\wt{f},v_{i}v_{j}}\lra{G,1}
 +\lra{\wt{f},1}\lra{G,v_{i}v_{j}}},\quad i\neq j,\\
 \lra{Q_{sym}(\wt{f},G),v_{i}v_{i}}
 =&-2b_{0}\lrc{\bblra{\wt{f},v_{i}v_{i}-\frac{|v|^{2}}{3}}
 \lra{G,1}
 +\lra{\wt{f},1}\bblra{G,v_{i}v_{i}-\frac{|v|^{2}}{3}}}.
\end{align*}
Therefore, we complete the proof of \eqref{equ:QfG,ij}--\eqref{equ:QfG,ii}.

Proceeding in a manner analogous to \eqref{equ:QfG,ij,proof}, we derive
\begin{align*}
\lra{Q(\wt{f},\wt{f}),v_{i}v_{j}}
=&-b_0\int_{\R^{3}\times \R^{3}}\wt{f}(v)\wt{f}(v_{*})\left[\left(v-v_*\right)_i\left(v-v_*\right)_j-\frac{\delta_{i j}}{3}\left|v-v_*\right|^2\right]dvdv_{*}.
\end{align*}
Utilizing the expansion formula \eqref{equ:vij,expansion}, we derive
\begin{align*}
\lra{Q(\wt{f},\wt{f}),v_{i}v_{j}}=&-2b_{0}\lrc{\lra{\wt{f},v_{i}v_{j}}
\lra{\wt{f},1}-\lra{\wt{f},v_{i}}\lra{\wt{f},v_{j}}},\quad i\neq j,\\
\lra{Q(\wt{f},\wt{f}),v_{i}v_{i}}=&-2b_{0}\lrc{\bblra{\wt{f},v_{i}v_{i}-\frac{|v|^{2}}{3}}
 \lra{\wt{f},1}+\lra{\wt{f},v_{i}}^{2}-\frac{1}{3}\sum_{j=1}^{3}\lra{\wt{f},v_{j}}^{2}},
\end{align*}
which completes the proof of \eqref{equ:Qff,ij}--\eqref{equ:Qff,ii}.

For \eqref{equ:QfG,ij,P0}--\eqref{equ:QfG,ii,P0}, using equations \eqref{equ:QfG,ij}--\eqref{equ:QfG,ii} and the conservation law \eqref{equ:mass,moment,conservation}, with the condition $\lra{G,1}=1$, we have
\begin{align*}
P_{0}^{x}\blra{Q_{sym}(\wt{f},G),v_{i}v_{j}}=
-2b_{0}P_{0}^{x}\lrc{\lra{\wt{f},v_{i}v_{j}}\lra{G,1}
 +\lra{\wt{f},1}\lra{G,v_{i}v_{j}}}=-2b_{0}P_{0}^{x}d_{ij},\ i\neq j,
\end{align*}
and
\begin{align*}
P_{0}^{x} \lra{Q_{sym}(\wt{f},G),v_{i}v_{i}}
 =&-2b_{0}P_{0}^{x}\lrc{  \bblra{\wt{f},v_{i}v_{i}-\frac{|v|^{2}}{3}}
 \lra{G,1}
 +\lra{\wt{f},1}\bblra{G,v_{i}v_{i}-\frac{|v|^{2}}{3}}}\\
 =&-2b_{0}\lrc{ P_{0}^{x}d_{ii}-\frac{1}{3}P_{0}^{x}E}.
\end{align*}
Hence, we have completed the proof of \eqref{equ:QfG,ij,P0}--\eqref{equ:QfG,ii,P0}.
\end{proof}

As an application of Lemmas \ref{lemma:evolution equation,moments} and \ref{lemma:cancellation,vij},we derive the evolution equation regarding the second-order moment by considering the zero-frequency mode
\begin{align}\label{equ:moment equation,second}
\frac{d}{dt}
\begin{pmatrix}
P_{0}^{x}E\\
P_{0}^{x}d_{12}\\
P_{0}^{x}d_{22}
\end{pmatrix}+\lrc{2\be I+
\begin{pmatrix}
0&2\al &0\\
0&2b_{0} & \al \\
-\frac{2b_{0}}{3}&0&2b_{0}
\end{pmatrix}}
\begin{pmatrix}
P_{0}^{x}E\\
P_{0}^{x}d_{12}\\
P_{0}^{x}d_{22}\\
\end{pmatrix}=
\begin{pmatrix}
0\\
P_{0}^{x}\lra{Q(\wt{f},\wt{f}),v_{1}v_{2}}\\
P_{0}^{x}\lra{Q(\wt{f},\wt{f}),v_{2}v_{2}}
\end{pmatrix}.
\end{align}
For notational convenience, we introduce
\begin{align}\label{equ:notation,U,A,R}
U=\begin{pmatrix}
P_{0}^{x}E\\
P_{0}^{x}d_{12}\\
P_{0}^{x}d_{22}
\end{pmatrix},\quad
A_{\al}=2\be I+
\begin{pmatrix}
0&2\al &0\\
0&2b_{0} & \al \\
-\frac{2b_{0}}{3}&0&2b_{0}
\end{pmatrix},\quad
R=\begin{pmatrix}
0\\
P_{0}^{x}\lra{Q(\wt{f},\wt{f}),v_{1}v_{2}}\\
P_{0}^{x}\lra{Q(\wt{f},\wt{f}),v_{2}v_{2}}
\end{pmatrix}.
\end{align}

We are now prepared to provide estimates for the second-order moment.
\begin{lemma}\label{lemma:L2,c,estimate}
Let $U(t)$ satisfy the equation \eqref{equ:moment equation,second}. Then we have
\begin{align}\label{equ:L2,U,estimate}
\sup_{s\in[0,t]}\n{U(s)}_{L_{x}^{2}}\leq \n{U(0)}_{L_{x}^{2}}+\n{g_{1}}_{\ol{X}_{1,t}}^{2}+\lrs{\sup_{s\in[0,t]}e^{\la_{0} s}\n{\pa_{x}g_{2}(s)}_{L_{x,v}^{2}}}^{2}.
\end{align}
In particular, we obtain
\begin{align}\label{equ:L2,c,estimate}
\sup_{s\in[0,t]}\n{c(s)}_{L_{x}^{2}}\lesssim & \n{w_{l}\wt{f}_{0}}_{L_{x,v}^{\infty}}+\n{g_{1}}_{\ol{X}_{1,t}}+\sup_{s\in[0,t]} e^{\la_{0}s}\n{\pa_{x}g_{2}(s)}_{L_{x}^{2}}\\
&+\n{g_{1}}_{\ol{X}_{1,t}}^{2}+\lrs{\sup_{s\in[0,t]}e^{\la_{0} s}\n{\pa_{x}g_{2}(s)}_{L_{x,v}^{2}}}^{2}.\notag
\end{align}
\end{lemma}
\begin{proof}
By Duhamel's formula, we rewrite
\begin{align*}
U(t)=e^{-A_{\al}t}U(0)+\int_{0}^{t}e^{-(t-s)A_{\al}}R(s)ds.
\end{align*}
We get into the analysis of the matrix $A_{\al}$. We compute the characteristic polynomial
\begin{align*}
\operatorname{det}(A_{\al}-\la)=:p(\la)=(2b_{0}+2\beta-\la)^{2}(2\be-\la)-\frac{4b_{0}\al^{2}}{3}.
\end{align*}
By \eqref{equ:Qff,ij}--\eqref{equ:Qff,ii} in Lemma \ref{lemma:cancellation,vij} with  $\wt{f}$ replaced by $G$, recalling that $\lra{G,1}=1$ and $\lra{G,v_{i}}=0$, we get
\begin{align*}
\lra{Q(G,G),v_{1}v_{2}}=-2b_{0}\lrc{\lra{G,v_{1}v_{2}}
\lra{G,1}-\lra{G,v_{1}}\lra{G,v_{2}}}=-2b_{0}\lra{G,v_{1}v_{2}},
\end{align*}
and
\begin{align*}
\lra{Q(G,G),v_{2}v_{2}}=&-2b_{0}\lrc{\bblra{G,v_{2}v_{2}-\frac{|v|^{2}}{3}}
 \lra{G,1}+\lra{G,v_{2}}^{2}-\frac{1}{3}\sum_{i=1}^{3}\lra{G,v_{i}}^{2}}\\
 =&-2b_{0}\lrs{ \lra{G,v_{2}v_{2}}-\bblra{G,\frac{|v|^{2}}{3}}}.
\end{align*}
Therefore, by the steady equation \eqref{equ:steady profile,G} for $G$, we have
\begin{align*}
A_{\al}
\begin{pmatrix}
\lra{G,|v|^{2}}\\
\lra{G,v_{1}v_{2}}\\
\lra{G,v_{2}v_{2}}
\end{pmatrix}=
\lrc{2\be I+
\begin{pmatrix}
0&2\al &0\\
0&2b_{0} & \al \\
-\frac{2b_{0}}{3}&0&2b_{0}
\end{pmatrix}}
\begin{pmatrix}
\lra{G,|v|^{2}}\\
\lra{G,v_{1}v_{2}}\\
\lra{G,v_{2}v_{2}}
\end{pmatrix}
=
\begin{pmatrix}
0\\
0\\
0
\end{pmatrix}.
\end{align*}
Noting that $\lra{|v|^{2},G}=3$, this implies that $\la_{1}=0$ is one of the eigenvalues of the matrix $A_{\al}$, that is,
\begin{align*}
p(0)=2\be(2b_{0}+2\beta)^{2}-\frac{4b_{0}\al^{2}}{3}=0,
\end{align*}
which yields the asymptotic behavior
\begin{align*}
\beta(\al)=\frac{\al^{2}}{6b_{0}}+O(\al^{4}).
\end{align*}
Thus, we can rewrite the characteristic polynomial as
\begin{align*}
p(\la)=-\la^{3}+2\la^{2}(2b_{0}+3\be)-4\la(b_{0}+\beta)(b_{0}+3\beta).
\end{align*}

A direct calculation shows that the other two eigenvalues are explicitly given by
\begin{equation*}
\left\{
\begin{aligned}
\la_{2}=&2b_{0}+3\beta+i\sqrt{b_{0}\beta+\frac{3}{4}\beta^{2}}= 2b_{0}+i\frac{\sqrt{6}\al}{6}+O(\al^{2}),\\
\la_{3}=&2b_{0}+3\beta-i\sqrt{b_{0}\beta+\frac{3}{4}\beta^{2}}= 2b_{0}-i\frac{\sqrt{6}\al}{6}+O(\al^{2}).
\end{aligned}
\right.
\end{equation*}
Consequently, the eigenvector matrix takes the form
\begin{equation*}
Q_{\al}=[q_{1},q_{2},q_{3}]=\begin{pmatrix}
\frac{\al^{2}}{\beta}& -\frac{2\al}{2\beta-\la_{2}}&-\frac{2\al}{2\beta-\la_{3}}\\
-\al& 1& 1\\
2b_{0}+2\beta& -\frac{2\beta+2b_{0}-\la_{2}}{\al}&  -\frac{2\beta+2b_{0}-\la_{3}}{\al}
\end{pmatrix},
\end{equation*}
with the asymptotic expansion
\begin{align*}
Q_{\al}=
\begin{pmatrix}
6b_{0}& 0& 0\\
0& 1 &1\\
2b_{0}& -i\frac{\sqrt{6}}{6}& i\frac{\sqrt{6}}{6}
\end{pmatrix}
+O(\al)=:Q_{0}+O(\al).
\end{align*}
Moreover, we have the diagonalization
\begin{align*}
Q_{\al}^{-1}A_{\al}Q_{\al}=\operatorname{Diag}(A_{\al})=
\begin{pmatrix}
0&&\\
&\la_{2}&\\
&&\la_{3}
\end{pmatrix}.
\end{align*}

We now derive a linear estimate for the semigroup $e^{-tA_{\al}}$. For $W=(w_{1},w_{2},w_{3})\in \R^{3}$, we have
\begin{align*}
\abs{e^{-A_{\al}t}W}=&\abs{Q_{\al}Q_{\al}^{-1}e^{-A_{\al}t}Q_{\al}Q_{\al}^{-1}W}=
\n{Q_{\al}e^{-\operatorname{Diag}(A_{\al})t}Q_{\al}^{-1}W}\\
\leq&\n{Q_{\al}}\n{Q_{\al}^{-1}}\abs{W}\lesssim \abs{W},
\end{align*}
where the last inequality follows from
\begin{align*}
\n{Q_{\al}}\lesssim 1,\quad \n{Q_{\al}^{-1}}\lesssim |\operatorname{det}Q_{\al}|^{-1}\lesssim |\operatorname{det}Q_{0}|^{-1}\lesssim 1.
\end{align*}
Then we get
\begin{align*}
\n{U(t)}_{L_{x}^{2}}\lesssim \n{U(0)}_{L_{x}^{2}}+\int_{0}^{t}\n{R(s)}_{L_{x}^{2}}ds,
\end{align*}
which yields that
\begin{align}\label{equ:estimate,U,duhamel,R}
\sup_{s\in[0,t]}\n{U(s)}_{L_{x}^{2}}\lesssim \n{U(0)}_{L_{x}^{2}}+\sup_{s\in[0,t]}e^{2\la_{0}s}\n{R(s)}_{L_{x}^{2}}.
\end{align}

Next, we get into the analysis of the nonlinear term $R(s)$ as given by \eqref{equ:notation,U,A,R}.
A direct approach, using the bilinear estimates \eqref{equ:loss term,estimate,f,g}--\eqref{equ:gain term,estimate,f,g} and the basic moment estimate \eqref{equ:moment estimate,basic} to bound
$R(s)$ in \eqref{equ:estimate,U,duhamel,R}, fails to provide decay, owing to the absence of decay estimates for the energy moment $E$.
However, a null structure identified in Lemma \ref{lemma:cancellation,vij} becomes crucial here.
Specifically, the mass and momentum conservation laws
\begin{align*}
P_{0}^{x}\lra{\wt{f},1}=P_{0}^{x}\lra{\wt{f},v_{i}}=0,\quad i=1,2,3,
\end{align*}
together with identities \eqref{equ:Qff,ij}--\eqref{equ:Qff,ii} in Lemma \ref{lemma:cancellation,vij}, imply the vanishing property of the zero-frequency mode in the nonlinear collision term. That is,
\begin{align*}
&P_{0}^{x}\lra{Q(\wt{f},\wt{f}),v_{1}v_{2}}\\
=&-2b_{0}P_{0}^{x}\lrc{\lra{\wt{f},v_{1}v_{2}}
\lra{\wt{f},1}-\lra{\wt{f},v_{1}}\lra{\wt{f},v_{2}}}\\
=&-2b_{0}P_{0}^{x}\lrc{\lrs{P_{\neq 0}^{x}\lra{\wt{f},v_{1}v_{2}}}
\lrs{P_{\neq 0}^{x}\lra{\wt{f},1}}-\lrs{P_{\neq 0}^{x}\lra{\wt{f},v_{1}}}
\lrs{P_{\neq 0}^{x}\lra{\wt{f},v_{2}}}},
\end{align*}
and
\begin{align*}
&P_{0}^{x}\lra{Q(\wt{f},\wt{f}),v_{2}v_{2}}\\
=&-2b_{0}P_{0}^{x}\lrc{\bblra{\wt{f},v_{2}v_{2}-\frac{|v|^{2}}{3}}
 \lra{\wt{f},1}+\lra{\wt{f},v_{2}}^{2}-\frac{1}{3}\sum_{i=1}^{3}\lra{\wt{f},v_{i}}^{2}}\\
 =&-2b_{0}P_{0}^{x}\lrc{\lrs{ P_{\neq 0}^{x}\bblra{\wt{f},v_{2}v_{2}-\frac{|v|^{2}}{3}}}
 \lrs{P_{\neq 0}^{x}\lra{\wt{f},1}}+\lrs{P_{\neq 0}^{x}\lra{\wt{f},v_{2}}}^{2}-\frac{1}{3}\sum_{i=1}^{3}\lrs{P_{\neq 0}^{x}\lra{\wt{f},v_{i}}}^{2}}.
\end{align*}
Then by Cauchy-Schwarz inequality, we have
\begin{align}\label{equ:estimate,duhamel,R123}
\sup_{s\in[0,t]}e^{2\la_{0}s}\n{R(s)}_{L_{x}^{2}}\lesssim R_{1}+R_{2}+R_{3},
\end{align}
where
\begin{align*}
R_{1}=&\sup_{s\in [0,t]}e^{\la_{0}s}\n{P_{\neq 0}^{x}
\lra{\wt{f}(s),v_{i}v_{j}}}_{L_{x}^{2}}
\sup_{s\in [0,t]}e^{\la_{0}s}\n{P_{\neq 0}^{x}
\lra{\wt{f}(s),1}}_{L_{x}^{2}},\\
R_{2}=&\sup_{s\in [0,t]}e^{\la_{0}s}\n{P_{\neq 0}^{x}
\lra{\wt{f}(s),|v|^{2}}}_{L_{x}^{2}}
\sup_{s\in [0,t]}e^{\la_{0}s}\n{P_{\neq 0}^{x}
\lra{\wt{f}(s),1}}_{L_{x}^{2}},\\
R_{3}=&\sup_{s\in [0,t]}e^{\la_{0}s}\n{P_{\neq 0}^{x}
\lra{\wt{f}(s),v}}_{L_{x}^{2}}
\sup_{s\in [0,t]}e^{\la_{0}s}\n{P_{\neq 0}^{x}
\lra{\wt{f}(s),v}}_{L_{x}^{2}}.
\end{align*}
By the Poincar\'e inequality and the moment estimate \eqref{equ:moment estimate,basic}, we get
\begin{align*}
\n{P_{\neq 0}^{x}\lra{\wt{f},p_{n}(v)}}_{L_{x}^{2}}\lesssim
\n{\lra{\pa_{x}\wt{f},p_{n}(v)}}_{L_{x}^{2}}\lesssim
 \n{w_{l}\pa_{x}g_{1}}_{L_{x,v}^{\infty}}+\n{\pa_{x}g_{2}}_{L_{x,v}^{2}},
\end{align*}
and hence obtain
\begin{align}\label{equ:estimate,Ri}
R_{i}\lesssim \lrs{\sup_{s\in [0,t]}e^{\la_{0}s}\n{w_{l}\pa_{x}g_{1}(s)}_{L_{x,v}^{\infty}}+\sup_{s\in [0,t]}e^{\la_{0}s}\n{\pa_{x}g_{2}(s)}_{L_{x,v}^{2}}}^{2},\quad i=1,2,3.
\end{align}
Putting together estimates \eqref{equ:estimate,U,duhamel,R}, \eqref{equ:estimate,duhamel,R123}, and \eqref{equ:estimate,Ri}, we arrive at the desired estimate \eqref{equ:L2,U,estimate}.

For \eqref{equ:L2,c,estimate}, by the triangle inequality, Poincar\'e inequality, and the moment estimate \eqref{equ:moment estimate,basic}, we get
\begin{align*}
\n{c}_{L_{x}^{2}}\lesssim& \n{P_{0}^{x}E}_{L_{x}^{2}}+\n{a}_{L_{x}^{2}}\lesssim \n{E}_{L_{x}^{2}}+\n{\pa_{x}a}_{L_{x}^{2}}\\
\lesssim&\n{U}_{L_{x}^{2}}+\n{w_{l}\pa_{x}g_{1}}_{L_{x,v}^{\infty}}+\n{\pa_{x}g_{2}}_{L_{x,v}^{2}},
\end{align*}
which, together with \eqref{equ:L2,U,estimate}, gives that
\begin{align*}
\sup_{s\in[0,t]}\n{c(s)}_{L_{x}^{2}}\lesssim & \n{w_{l}\wt{f}_{0}}_{L_{x,v}^{\infty}}+\n{g_{1}}_{\ol{X}_{1,t}}^{2}+\lrs{\sup_{s\in[0,t]}e^{\la_{0} s}\n{\pa_{x}g_{2}(s)}_{L_{x,v}^{2}}}^{2}\\
&+ \n{g_{1}}_{\ol{X}_{1,t}}+\sup_{s\in[0,t]} e^{\la_{0}s}\n{\pa_{x}g_{2}(s)}_{L_{x}^{2}}.
\end{align*}
This completes the proof of \eqref{equ:L2,c,estimate}.
\end{proof}

In summary, combining Lemma \ref{lemma:g2,L2,estimate,p0,p1} and Lemma \ref{lemma:L2,c,estimate}, we arrive at the following $L_{x,v}^{2}$ energy estimates for $g_{2}$.

\begin{proposition}\label{lemma:g2,L2,estimate}
Let $[g_{1},g_{2}]$ be the solution to \eqref{equ:g1,equation}--\eqref{equ:g2,equation}. Then we have
\begin{align}\label{equ:g2,L2,estimate,full}
\sup_{s\in[0,t]}\n{g_{2}(s)}_{L_{x,v}^{2}}
\lesssim&
 \sum_{k=0,1}\n{w_{l}\pa_{x}^{k}\wt{f}_{0}}_{L_{x,v}^{\infty}}+ \al \sum_{k=0,1}\n{g_{2}}_{\ol{X}_{k,t}}+\sum_{k=0,1}\n{[g_{1},g_{2}]}_{\ol{X}_{k,t}}^{2}
\\
 &+
 \sup_{s\in[0,t]} e^{\la_{0} s} \n{\pa_{x}g_{2}(s)}_{L_{x,v}^{2}}+\lrs{\sup_{s\in[0,t]}e^{\la_{0} s}\n{\pa_{x}g_{2}(s)}_{L_{x,v}^{2}}}^{2}.\notag
\end{align}

\end{proposition}
\begin{proof}
By Lemmas \ref{lemma:g2,L2,estimate,p0,p1}, \ref{lemma:L2,c,estimate}, and \ref{lemma,Linf,estimate}, we obtain
\begin{align*}
\sup_{s\in[0,t]}\n{g_{2}(s)}_{L_{x,v}^{2}}\lesssim&
 \sup_{s\in[0,t]} \n{c(s)}_{L_{x}^{2}}+
\sum_{k=0,1} \n{g_{1}}_{\ol{X}_{k,t}}+\sup_{s\in[0,t]} e^{\la_{0}s}\n{\pa_{x}g_{2}(s)}_{L_{x}^{2}},\\
\sup_{s\in[0,t]}\n{c(s)}_{L_{x}^{2}}\lesssim & \n{w_{l}\wt{f}_{0}}_{L_{x,v}^{\infty}}+\n{g_{1}}_{\ol{X}_{1,t}}+\sup_{s\in[0,t]} e^{\la_{0}s}\n{\pa_{x}g_{2}(s)}_{L_{x}^{2}}\\
&+\n{g_{1}}_{\ol{X}_{1,t}}^{2}+\lrs{\sup_{s\in[0,t]}e^{\la_{0} s}\n{\pa_{x}g_{2}(s)}_{L_{x,v}^{2}}}^{2},\notag\\
\sum_{k=0,1}\n{g_{1}}_{\ol{X}_{k,t}}\lesssim &\sum_{k=0,1}\n{w_{l}\pa_{x}^{k}\wt{f}_{0}}_{L_{x,v}^{\infty}}+\al \sum_{k=0,1}\n{g_{2}}_{\ol{X}_{k,t}}+\sum_{k=0,1}\n{[g_{1},g_{2}]}_{\ol{X}_{k,t}}^{2},
\end{align*}
which suffices to give the desired estimate \eqref{equ:g2,L2,estimate,full}.
\end{proof}

\subsection{Higher Energy Estimates}\label{sec:Higher Energy Estimates}
In this subsection, we focus on deriving higher energy estimates. These estimates are essential for analyzing the long-time decay properties of solutions to the coupled system \eqref{equ:g1,equation}--\eqref{equ:g2,equation} at higher energy levels, and complementing the lower energy estimates established in Proposition \ref{lemma:g2,L2,estimate}.
\begin{lemma}\label{lemma:g2,higher energy}
Let $[g_{1},g_{2}]$ be the solution to \eqref{equ:g1,equation}--\eqref{equ:g2,equation}. There holds that
\begin{align}\label{equ:g2,higher energy}
\sup_{s\in[0,t]} e^{\lambda_{0} s}\n{\pa_{x} g_{2}(s)}_{L_{x,v}^{2}}\lesssim \n{w_{l}\pa_{x}\wt{f}_{0}}_{L_{x,v}^{\infty}}+
\al \n{g_{2}}_{\ol{X}_{1,t}}+
 \lrs{\sum_{k=0,1}\n{[g_{1},g_{2}]}_{\ol{X}_{k,t}}}^{2}.
\end{align}
\end{lemma}
\begin{proof}
We recall that
\begin{align}\label{equ:g2,equation,energy estimate}
\pa_t \pa_{x}g_{2}+e^{\beta t} \sum_{i=1}^{3}v_{i}\pa_{x_{i}} \pa_{x}g_{2}-\beta \nabla_{v} \cdot\left(v \pa_{x}g_{2}\right)-\alpha v_{2} \pa_{v_{1}} \pa_{x}g_{2}+L \pa_{x}g_{2} =\mu^{-1 / 2}\chi_{M}K_{\mu} \pa_{x}g_{1}.
\end{align}
Testing \eqref{equ:g2,equation,energy estimate} by $\pa_{x}g_{2}$, we obtain
\begin{align*}
&\lra{\pa_t \pa_{x}g_{2}+e^{\beta t} \sum_{i=1}^{3}v_{i}\pa_{x_{i}} \pa_{x}g_{2}-\beta \nabla_{v} \cdot\left(v \pa_{x}g_{2}\right)-\alpha v_{2} \pa_{v_{1}} \pa_{x}g_{2}+L \pa_{x}g_{2} ,\pa_{x}g_{2}}\\
=&\lra{\mu^{-1 / 2}\chi_{M}K_{\mu} \pa_{x}g_{1},\pa_{x}g_{2}}.
\end{align*}
Following the same argument as used in \eqref{equ:micro estimate,lower energy,proof},
we have
\begin{align}\label{equ:micro estimate,higher energy,proof}
\frac{d}{dt}\n{\pa_{x}g_{2}}_{L_{x,v}^{2}}^{2}+\delta_{0}\n{\mathbf{P}_{1}\pa_{x}g_{2}}_{L_{x,v}^{2}}^{2}\leq \ve\n{\mathbf{P}_{0}\pa_{x}g_{2}}_{L_{x,v}^{2}}^{2}
+C_{\ve,M}\n{w_{l}\pa_{x}g_{1}}_{L_{x,v}^{\infty}}^{2}.
\end{align}

With the basic energy estimate in hand, we now focus on estimating the macroscopic components.
To achieve this, we employ the macro-micro decomposition
\begin{align*}
d_{ij}=\mathbf{P}_{0}d_{ij}+\mathbf{P}_{1}d_{ij},\quad
B_{i}=\mathbf{P}_{0}B_{i}+\mathbf{P}_{1}B_{i},\quad B_{ij}=\mathbf{P}_{0}B_{ij}+\mathbf{P}_{1}B_{ij},
\end{align*}
where
\begin{equation}\label{equ:notation,P0P1,dij}
\left\{
\begin{aligned}
&\mathbf{P}_{0}d_{ij}:=\lra{v_{i}v_{j},\sqrt{\mu}\mathbf{P}_{0}\wt{g}},\quad \mathbf{P}_{1}d_{ij}:=\lra{v_{i}v_{j},\sqrt{\mu}\mathbf{P}_{1}\wt{g}},\\
&\mathbf{P}_{0}B_{i}:=\lra{v_{i}(|v|^{2}-5),\sqrt{\mu}\mathbf{P}_{0}\wt{g}},\quad \mathbf{P}_{1}B_{i}:=\lra{v_{i}(|v|^{2}-5),\sqrt{\mu}\mathbf{P}_{1}\wt{g}},\\
&\mathbf{P}_{0}B_{ij}:=\lra{v_{i}v_{j}(|v|^{2}-5),\sqrt{\mu}\mathbf{P}_{0}\wt{g}},\quad \mathbf{P}_{1}B_{ij}:=\lra{v_{i}v_{j}(|v|^{2}-5),\sqrt{\mu}\mathbf{P}_{1}\wt{g}}.
\end{aligned}
\right.
\end{equation}
A direct calculation yields
\begin{align*}
\mathbf{P}_{0}d_{ij}=\delta_{ij}\lrs{a+\frac{\sqrt{6}}{3}c},\quad
\mathbf{P}_{0}B_{i}=0,\quad
\mathbf{P}_{0}B_{ij}=\delta_{ij}\frac{5\sqrt{6}}{3}c.
\end{align*}

Next, by the evolution equations of moments in Lemma \ref{lemma:evolution equation,moments}, we derive
\begin{align}
\pa_{t}a+e^{\beta t}\sum_{i=1}^{3}\pa_{x_{i}}b_{i}=&0,\label{equ:moment,equation,a,P01}\\
\pa_{t}b_{i}+e^{\beta t}\pa_{x_{i}}a=&\mathcal{R}_{\al}(b_{i}),\label{equ:moment,equation,bi,P01}\\
\pa_{t}c+\frac{\sqrt{6}}{3}e^{\beta t}\sum_{i=1}^{3}\pa_{x_{i}}b_{i}=&\mathcal{R}_{\al}(c),\label{equ:moment,equation,c,P01}\\
\pa_{t}d_{ij}+e^{\beta t}\pa_{x_{i}}b_{j}+e^{\beta t}\pa_{x_{j}}b_{i}=&\mathcal{R}_{\al}(d_{ij}),\ i\neq j,\label{equ:moment,equation,dij,P01}\\
\frac{3}{2}\pa_{t}\mathbf{P}_{1}d_{jj}+3e^{\beta t}\pa_{x_{j}}b_{j}-e^{\beta t}\sum_{i=1}^{3}\pa_{x_{i}}b_{i}=&
\mathcal{R}_{\al}(\mathbf{P}_{1}d_{jj}),\label{equ:moment,equation,djj,P01}\\
\pa_{t}\lrs{\frac{3}{2}\pa_{x_{j}}\mathbf{P}_{1}d_{jj}+\sum_{i\neq j}\pa_{x_{i}}d_{ij}}
-e^{\beta t}\sum_{i=1}^{3}\pa_{x_{i}}^{2}b_{j}-e^{\beta t}\pa_{x_{j}}^{2}b_{j}=&\mathcal{R}_{\al}(\pa_{x_{j}}\mathbf{P}_{1}d_{jj},\sum_{i\neq j}\pa_{x_{i}}d_{ij}),\label{equ:moment,equation,djj,dij,P01}\\
\pa_{t}B_{i}+\frac{5\sqrt{6}}{3}e^{\beta t}\pa_{x_{i}}c=&\mathcal{R}_{\al}(B_{i}),\label{equ:moment,equation,B1,P01}
\end{align}
where the remainder terms $\mathcal{R}_{\al}(\cdot)$ and $\mathcal{R}_{\al}(\cdot,\cdot)$ consist of terms involving the smallness factor $\al$, the microscopic part, and the nonlinear part.

To provide a coercive estimate for the macroscopic part, we introduce the interaction functional
\begin{align*}
\mathcal{E}_{int}(t):=
2\mathcal{E}_{int,a}(t)+10\mathcal{E}_{int,\mathbf{b}}(t)+10\mathcal{E}_{int,c}(t),
\end{align*}
where
\begin{equation*}
\left\{
\begin{aligned}
\mathcal{E}_{int,a}(t):=&\sum_{i=1}^{3}\lra{b_{i},\pa_{x_{i}}a},\\
\mathcal{E}_{int,\mathbf{b}}(t):=&\sum_{j=1}^{3}\bblra{\frac{3}{2}\pa_{x_{j}}\mathbf{P}_{1}d_{jj}+\sum_{i\neq j}\pa_{x_{i}}d_{ij},b_{j}},\\
\mathcal{E}_{int,c}(t):=&\sum_{i=1}^{3}\lra{B_{i},\pa_{x_{i}}c}.
\end{aligned}
\right.
\end{equation*}

We now derive functional inequalities for the interaction functional.
Starting with $\mathcal{E}_{int,a}(t)$, by \eqref{equ:moment,equation,a,P01}, \eqref{equ:moment,equation,bi,P01}, and Young's inequality, we get
\begin{align*}
\frac{d}{dt}\mathcal{E}_{int,a}(t)+e^{\beta t}\n{\pa_{x}a}_{L_{x}^{2}}^{2}
=&-\sum_{i=1}^{3}\lra{\pa_{x_{i}}b_{i},\pa_{t}a}
+\sum_{i=1}^{3}\lra{\mathcal{R}_{\al}(b_{i}),\pa_{x_{i}}a}\\
\leq&\ve e^{\beta t}\n{\pa_{x}a}_{L_{x}^{2}}^{2}+5e^{\beta t}\sum_{i=1}^{3}\n{\pa_{x}b_{i}}_{L_{x}^{2}}^{2}+
C_{\ve}e^{-\beta t}\sum_{i=1}^{3}\n{\mathcal{R}_{\al}(b_{i})}_{L_{x}^{2}}^{2}.
\end{align*}
Similarly, for $\mathcal{E}_{int,\mathbf{b}}(t)$, by \eqref{equ:moment,equation,bi,P01} and \eqref{equ:moment,equation,djj,dij,P01}, we obtain
\begin{align*}
&\frac{d}{dt}\mathcal{E}_{int,\mathbf{b}}(t)+e^{\beta t}\n{\pa_{x}\mathbf{b}}_{L_{x}^{2}}^{2}\\
=&\sum_{j=1}^{3}\bblra{\frac{3}{2}\pa_{x_{j}}\mathbf{P}_{1}d_{jj}+\sum_{i\neq j}\pa_{x_{i}}d_{ij},\pa_{t}b_{j}}+\sum_{j=1}^{3}
\bblra{\mathcal{R}_{\al}(\pa_{x_{j}}\mathbf{P}_{1}d_{jj},\sum_{i\neq j}\pa_{x_{i}}d_{ij}),b_{j}}\\
\leq&  \ve e^{\beta t}\lrs{\n{\pa_{x}a}_{L_{x}^{2}}^{2}+\n{\pa_{x}\mathbf{b}}_{L_{x}^{2}}^{2}}+C_{\ve}e^{\beta t}\sum_{j=1}^{3}
\lrs{\n{\pa_{x}\mathbf{P}_{1}d_{jj}}_{L_{x}^{2}}^{2}+\sum_{i\neq j}\n{\pa_{x}d_{ij}}_{L_{x}^{2}}^{2}}\\
&+C_{\ve}e^{-\beta t}\sum_{j=1}^{3}\lrs{\n{P_{\neq 0}^{x}\mathcal{R}_{\al}(b_{j})}_{L_{x}^{2}}^{2}
+\n{P_{\neq 0}^{x}\mathcal{R}_{\al}(\pa_{x_{j}}\mathbf{P}_{1}d_{jj},\sum_{i\neq j}\pa_{x_{i}}d_{ij})}_{L_{x}^{2}}^{2}}.
\end{align*}
Lastly, for $\mathcal{E}_{int,c}(t)$, by \eqref{equ:moment,equation,c,P01} and \eqref{equ:moment,equation,B1,P01}, we have
\begin{align*}
&\frac{d}{dt}\mathcal{E}_{int,c}(t)+e^{\beta t}\n{\pa_{x}c}_{L_{x}^{2}}^{2}\\
=&-\sum_{i=1}^{3}\lra{\pa_{x_{i}}B_{i},\pa_{t}c}+\sum_{i=1}^{3}
\lra{\mathcal{R}_{\al}(B_{i}),\pa_{x_{i}}c}\\
=&\sum_{i=1}^{3}\bblra{\pa_{x_{i}}B_{i},\frac{\sqrt{6}}{3}e^{\beta t}\sum_{j=1}^{3}\pa_{x_{j}}b_{j}-\mathcal{R}_{\al}(c)}+\sum_{i=1}^{3}
\lra{\mathcal{R}_{\al}(B_{i}),\pa_{x_{i}}c}\\
\leq& \ve e^{\beta t}\lrs{\n{\pa_{x}\mathbf{b}}_{L_{x}^{2}}^{2}+\n{\pa_{x}c}_{L_{x}^{2}}^{2}}+C_{\ve}e^{\beta t}\sum_{i=1}^{3}
\n{\pa_{x_{i}}B_{i}}_{L_{x}^{2}}^{2}\\
&+C_{\ve}e^{-\beta t}\sum_{i=1}^{3}\lrs{\n{P_{\neq 0}^{x}\mathcal{R}_{\al}(c)}_{L_{x}^{2}}^{2}+\n{P_{\neq 0}^{x}\mathcal{R}_{\al}(B_{i})}_{L_{x}^{2}}^{2}}.
\end{align*}

Summing up the estimates for $\mathcal{E}_{int,a}(t)$, $\mathcal{E}_{int,\mathbf{b}}(t)$, and $\mathcal{E}_{int,c}(t)$, we conclude
\begin{align*}
&e^{-\beta t}\frac{d}{dt}\mathcal{E}_{int}+\n{\pa_{x}[a,b,c]}_{L_{x}^{2}}^{2}\leq C_{\ve}I+C_{\ve}I\!I
\end{align*}
where
\begin{align*}
I=&\sum_{i=1}^{3}
\n{\pa_{x_{i}}B_{i}}_{L_{x}^{2}}^{2}+\sum_{j=1}^{3}
\lrs{\n{\pa_{x}\mathbf{P}_{1}d_{jj}}_{L_{x}^{2}}^{2}+\sum_{i\neq j}\n{\pa_{x}d_{ij}}_{L_{x}^{2}}^{2}},\\
I\!I=&e^{-2\beta t}\sum_{i=1}^{3}\sum_{\zeta}\n{P_{\neq 0}^{x}\mathcal{R}_{\al}(\zeta)}_{L_{x}^{2}}^{2},
\end{align*}
with $\zeta\in \lr{c,b_{j},B_{j},[\pa_{x_{j}}\mathbf{P}_{1}d_{11},\sum_{i\neq j}\pa_{x_{i}}d_{ij}]}$.

For the term $I$, we use a similar approach as in the moment estimate \eqref{equ:moment estimate,basic}. For any polynomial $p_{n}(v)$, we have
\begin{align}\label{equ:moment estimate,P1}
\bn{\lra{\mu^{\frac{1}{2}}p_{n}(v),\pa_{x}\mathbf{P}_{1}\wt{g}}}_{L_{x}^{2}}\leq & \bn{\lra{\mu^{\frac{1}{2}}p_{n}(v),\mathbf{P}_{1}(\mu^{-\frac{1}{2}}\pa_{x}g_{1})}}_{L_{x}^{2}}+
\bn{\lra{\mu^{\frac{1}{2}}p_{n}(v),\mathbf{P}_{1}(\pa_{x}g_{2})}}_{L_{x}^{2}}\\
\lesssim& \n{w_{l}\pa_{x}g_{1}}_{L_{x}^{\infty}L_{v}^{\infty}}+\n{\pa_{x}\mathbf{P}_{1}g_{2}}_{L_{x,v}^{2}}.\notag
\end{align}
Noting that $B_{j}=\mathbf{P}_{1}B_{j}$, $d_{ij}=\mathbf{P}_{1}d_{ij}$ for $i\neq j$, by the definition \eqref{equ:notation,P0P1,dij} and the above estimate \eqref{equ:moment estimate,P1},
we hence obtain
\begin{align}\label{equ:estimate,I,higher}
I\lesssim& \n{w_{l}\pa_{x}g_{1}}_{L_{x,v}^{\infty}}^{2}+\n{\pa_{x}\mathbf{P}_{1}g_{2}}_{L_{x,v}^{2}}^{2}.
\end{align}

For the term $I\!I$, we only deal with the $\zeta=B_{1}$ case, as the other cases follow similarly.
By the moment equation \eqref{equmomentequationB1} and \eqref{equ:moment,equation,B1,P01}, we have
\begin{align}\label{equ:R,alpha,B1,proof}
P_{\neq 0}^{x}\mathcal{R}_{\al}(B_{1})=&P_{\neq 0}^{x}\lrs{-3\beta B_{1}-10\beta b_{1}-3\al d_{112}-e^{\beta t}\sum_{i=1}^{3} \pa_{x_{i}} \mathbf{P}_{1}B_{1i}+
\lra{v_{1}(|v|^{2}-5),\wt{\mathcal{Q}}}}.
\end{align}

For the linear terms in \eqref{equ:R,alpha,B1,proof},
by the moment estimates \eqref{equ:moment estimate,basic} and \eqref{equ:moment estimate,P1}, we get
\begin{align}
\beta \n{P_{\neq 0}^{x}B_{1}}_{L_{x}^{2}}+\beta \n{P_{\neq 0}^{x}b_{1}}_{L_{x}^{2}}+\al \n{P_{\neq 0}^{x}d_{112}}_{L_{x}^{2}}\lesssim& \al
\n{w_{l}\pa_{x}g_{1}}_{L_{x,v}^{\infty}}+\al \n{w_{l}\pa_{x}g_{2}}_{L_{x,v}^{\infty}},\label{equ:R,alpha,B1,beta}
\end{align}
and
\begin{align}
e^{\beta t}\sum_{i=1}^{3} \n{\pa_{x_{i}} \mathbf{P}_{1}B_{1i}}_{L_{x}^{2}}\lesssim &e^{\beta t}
\n{w_{l}\pa_{x}g_{1}}_{L_{x,v}^{\infty}}+e^{\beta t} \n{\pa_{x}\mathbf{P}_{1}g_{2}}_{L_{x,v}^{2}}.\label{equ:R,alpha,B1,P1}
\end{align}

For the force term $\lra{v_{1}|v^{2}|-5,\wt{Q}}$ in \eqref{equ:R,alpha,B1,proof}, we recall that
\begin{align*}
\wt{\mathcal{Q}}=Q(\wt{f},\wt{f})+Q_{sym}(\wt{f},\mu)+\al Q_{sym}(\wt{f},\mu^{\frac{1}{2}}G_{1}).
\end{align*}
We then estimate each term separately.
By Lemma \ref{lemma:Q,bilinear estimate,Lp}, the moment estimates \eqref{equ:moment estimate,basic}, and \eqref{equ:moment estimate,P1}, we have
\begin{align}\label{equ:R,alpha,B1,Qff}
&\n{P_{\neq 0}^{x}\blra{v_{1}(|v|^{2}-5),Q(\wt{f},\wt{f})}}_{L_{x}^{2}}\\
\lesssim & \n{Q_{sym}(\lra{v}^{3}|\pa_{x}\wt{f}|,\lra{v}^{3}|\wt{f}|)}_{L_{x}^{2}L_{v}^{1}}
+\n{Q_{sym}(\lra{v}^{3}|\wt{f}|,\lra{v}^{3}|\pa_{x}\wt{f}|)}_{L_{x}^{2}L_{v}^{1}}\notag\\
\lesssim& \n{\lra{v}^{3}\pa_{x}\wt{f}}_{L_{x}^{4}L_{v}^{1}}\n{\lra{v}^{3}\wt{f}}_{L_{x}^{4}L_{v}^{1}}
+\n{\lra{v}^{3}\pa_{x}\wt{f}}_{L_{x}^{4}L_{v}^{1}}\n{\lra{v}^{3}\wt{f}}_{L_{x}^{4}L_{v}^{1}}\notag\\
\lesssim& \lrs{\n{w_{l}g_{1}}_{L_{x,v}^{\infty}}+\n{w_{l}g_{2}}_{L_{x,v}^{\infty}}}
\lrs{\n{w_{l}\pa_{x}g_{1}}_{L_{x,v}^{\infty}}+\n{w_{l}\pa_{x}g_{2}}_{L_{x,v}^{\infty}}},\notag
\end{align}
and
\begin{align}\label{equ:R,alpha,B1,Qfm}
&\n{P_{\neq 0}^{x}\blra{v_{1}(|v|^{2}-5),Q_{sym}(\wt{f},\mu)}}_{L_{x}^{2}}\\
\leq &\n{\blra{v_{1}(|v|^{2}-5),Q_{sym}(\pa_{x}\wt{f},\mu)}}_{L_{x}^{2}}\notag\\
\leq&
\n{\blra{v_{1}(|v|^{2}-5),Q_{sym}(\pa_{x}g_{1},\mu)}}_{L_{x}^{2}}
+\n{\blra{v_{1}(|v|^{2}-5)\mu^{\frac{1}{2}},L\pa_{x}g_{2}}}_{L_{x}^{2}}\notag\\
\lesssim&\n{w_{l}\pa_{x}g_{1}}_{L_{x,v}^{\infty}}+\n{\pa_{x}\mathbf{P}_{1}g_{2}}_{L_{x,v}^{2}},\notag
\end{align}
and
\begin{align}\label{equ:R,alpha,B1,QfG1}
\al \n{P_{\neq 0}^{x}\blra{v_{1}(|v|^{2}-5),Q_{sym}(\wt{f},\mu^{\frac{1}{2}}G_{1})}}_{L_{x}^{2}}
\lesssim& \al\n{w_{l}\pa_{x}g_{1}}_{L_{x,v}^{\infty}}+\al \n{w_{l}\pa_{x}g_{2}}_{L_{x,v}^{\infty}}.
\end{align}
Combining the estimates \eqref{equ:R,alpha,B1,beta}, \eqref{equ:R,alpha,B1,P1}, \eqref{equ:R,alpha,B1,Qff}, \eqref{equ:R,alpha,B1,Qfm}, and \eqref{equ:R,alpha,B1,QfG1}, we arrive at
\begin{align*}
&e^{-2\beta t}\n{P_{\neq 0}^{x}\mathcal{R}_{\al}(B_{1})}_{L_{x}^{2}}^{2}\\
\lesssim &
\n{\pa_{x}\mathbf{P}_{1}g_{2}}_{L_{x,v}^{2}}^{2}+
\lrs{\n{w_{l}g_{1}}_{L_{x,v}^{\infty}}+\n{w_{l}g_{2}}_{L_{x,v}^{\infty}}}^{2}
\lrs{\n{w_{l}\pa_{x}g_{1}}_{L_{x,v}^{\infty}}+\n{w_{l}\pa_{x}g_{2}}_{L_{x,v}^{\infty}}}^{2}\\
&+\n{w_{l}\pa_{x}g_{1}}_{L_{x,v}^{\infty}}^{2}+\al^{2} \n{w_{l}\pa_{x}g_{2}}_{L_{x,v}^{\infty}}^{2}.
\end{align*}
Therefore, we arrive at
\begin{align}\label{equ:estimate,II,higher}
I\!I\lesssim& \n{\pa_{x}\mathbf{P}_{1}g_{2}}_{L_{x,v}^{2}}^{2}+
\lrs{\n{w_{l}g_{1}}_{L_{x,v}^{\infty}}+\n{w_{l}g_{2}}_{L_{x,v}^{\infty}}}^{2}
\lrs{\n{w_{l}\pa_{x}g_{1}}_{L_{x,v}^{\infty}}+\n{w_{l}\pa_{x}g_{2}}_{L_{x,v}^{\infty}}}^{2}\\
&+\n{w_{l}\pa_{x}g_{1}}_{L_{x,v}^{\infty}}^{2}+\al^{2} \n{w_{l}\pa_{x}g_{2}}_{L_{x,v}^{\infty}}^{2}.\notag
\end{align}

Combining estimates \eqref{equ:estimate,I,higher} for $I$ and \eqref{equ:estimate,II,higher} for $I\! I$, we reach
\begin{align}\label{equ:abc,higher energy,proof}
&e^{-\beta t}\frac{d}{dt}\mathcal{E}_{int}+\n{\pa_{x}[a,b,c]}_{L_{x}^{2}}^{2}\\
\lesssim &
\n{\pa_{x}\mathbf{P}_{1}g_{2}}_{L_{x,v}^{2}}^{2}+
\lrs{\n{w_{l}g_{1}}_{L_{x,v}^{\infty}}+\n{w_{l}g_{2}}_{L_{x,v}^{\infty}}}^{2}
\lrs{\n{w_{l}\pa_{x}g_{1}}_{L_{x,v}^{\infty}}+\n{w_{l}\pa_{x}g_{2}}_{L_{x,v}^{\infty}}}^{2}\notag\\
&+\n{w_{l}\pa_{x}g_{1}}_{L_{x,v}^{\infty}}^{2}+\al^{2} \n{w_{l}\pa_{x}g_{2}}_{L_{x,v}^{\infty}}^{2}.\notag
\end{align}

Then we have
\begin{align}\label{equ:higher energy,energy estimate,proof}
&\frac{d}{dt}\lrs{\n{\pa_{x} g_{2}}_{L_{x,v}^{2}}^{2}+3\la_{0}e^{-\beta t}\mathcal{E}_{int}}+
\delta_{0}\n{\pa_{x} \mathbf{P}_{1} g_{2}}_{L_{x,v}^{2}}^{2}+3\la_{0}\n{\pa_{x}\mathbf{P}_{0}g_{2}}_{L_{x,v}^{2}}^{2}\\
=&\frac{d}{dt}\n{\pa_{x} g_{2}}_{L_{x,v}^{2}}^{2}+3\la_{0}e^{-\beta t}\frac{d}{dt}\mathcal{E}_{int}+
3\la_{0}\n{\pa_{x} \mathbf{P}_{1} g_{2}}_{L_{x,v}^{2}}^{2}+\delta_{0}\n{\pa_{x}\mathbf{P}_{0}g_{2}}_{L_{x,v}^{2}}^{2}-3\la_{0}\beta e^{-\beta t}\mathcal{E}_{int}.\notag
\end{align}
Putting \eqref{equ:micro estimate,higher energy,proof} and $\eqref{equ:abc,higher energy,proof}$ into \eqref{equ:higher energy,energy estimate,proof},
and using estimates that
\begin{align*}
\ve\n{\pa_{x}\mathbf{P}_{0}g_{2}}_{L_{x,v}^{2}}\leq& \ve\n{\pa_{x}[a,b,c]}_{L_{x}^{2}}+ \ve \n{\pa_{x}w_{l}g_{1}}_{L_{x,v}^{\infty}},\\
\beta |e^{-\beta t}\mathcal{E}_{int}(t)|\lesssim & \al^{2}\n{w_{l}\pa_{x}g_{1}(t)}_{L_{x,v}^{\infty}}^{2}+\al^{2}\n{w_{l}\pa_{x}g_{2}(t)}_{L_{x,v}^{\infty}}^{2},
\end{align*}
with $\la_{0}\ll \delta_{0}$ and $\ve\ll \la_{0}$, we arrive at
\begin{align*}
&\frac{d}{dt}\lrs{\n{\pa_{x} g_{2}}_{L_{x,v}^{2}}^{2}+3\la_{0}e^{-\beta t}\mathcal{E}_{int}}+
2\lambda_0\lrs{\n{\pa_{x} \mathbf{P}_{1} g_{2}}_{L_{x,v}^{2}}^{2}+\n{\pa_{x}\mathbf{P}_{0}g_{2}}_{L_{x,v}^{2}}^{2}}\\
\lesssim&\lrs{\n{w_{l}g_{1}}_{L_{x,v}^{\infty}}+\n{w_{l}g_{2}}_{L_{x,v}^{\infty}}}^{2}
\lrs{\n{w_{l}\pa_{x}g_{1}}_{L_{x,v}^{\infty}}+\n{w_{l}\pa_{x}g_{2}}_{L_{x,v}^{\infty}}}^{2}\notag\\
&+\n{w_{l}\pa_{x}g_{1}}_{L_{x,v}^{\infty}}^{2}+\al^{2} \n{w_{l}\pa_{x}g_{2}}_{L_{x,v}^{\infty}}^{2}.
\end{align*}

Applying Gronwall's inequality and absorbing the term $3\la_{0}\sup_{s\in[0,t]} e^{\la_{0}s}|e^{-\beta s}\mathcal{E}_{int}(s)|$, we arrive at
\begin{align*}
&\sup_{s\in[0,t]} e^{\la_{0}s}\n{\pa_{x} g_{2}(s)}_{L_{x,v}^{2}}\\
 \lesssim & \sup_{s\in[0,t]} e^{\la_{0}s}\n{w_{l} \pa_{x} g_{1}(s)}_{L_{x,v}^{\infty}} + \alpha \sup_{s\in[0,t]} e^{\la_{0}s}\n{w_{l} \pa_{x} g_{2}(s)}_{L_{x,v}^{\infty}}\notag\\
 &+\sup_{s\in[0,t]}e^{\la_{0}s}\lrs{\n{w_{l}g_{1}}_{L_{x,v}^{\infty}}+\n{w_{l}g_{2}}_{L_{x,v}^{\infty}}}
 \lrs{\n{w_{l}\pa_{x}g_{1}}_{L_{x,v}^{\infty}}+\n{w_{l}\pa_{x}g_{2}}_{L_{x,v}^{\infty}}}\notag\\
 \lesssim& \n{g_{1}}_{\ol{X}_{1,t}}+\al \n{g_{2}}_{\ol{X}_{1,t}}+
\lrs{\n{[g_{1},g_{2}]}_{X_{0,t}}}\lrs{\n{[g_{1},g_{2}]}_{\ol{X}_{1,t}}}\notag\\
 \lesssim& \n{w_{l}\pa_{x}\wt{f}_{0}}_{L_{x,v}^{\infty}}+\al \n{g_{2}}_{\ol{X}_{1,t}}+
 \lrs{\sum_{k=0,1}\n{[g_{1},g_{2}]}_{\ol{X}_{k,t}}}^{2},\notag
\end{align*}
where in the last inequality we used Lemma \ref{lemma,Linf,estimate} to absorb $\n{g_{1}}_{\ol{X}_{1,t}}$. Hence, we complete the proof of \eqref{equ:g2,higher energy}.
\end{proof}

\section{Global Stability and the Energy Growth}\label{sec:Global Stability}
In the section, we focus on proving the main Theorem \ref{thm:main theorem} which concerns the global stability, and provide a detailed analysis of the long-time behavior of the energy growth.
\begin{proof}[\textbf{Proof of Theorem \ref{thm:main theorem}}]

Combining the results from \eqref{equ:Linf,L2,estimate} in Lemma \ref{lemma,Linf,estimate} and \eqref{equ:g2,L2,estimate,full} in Lemma \ref{lemma:g2,L2,estimate}, we derive
\begin{align}\label{equ:main proof,X,L2}
&\sum_{k=0,1}\n{[g_{1},g_{2}]}_{\ol{X}_{k,t}} \\
\lesssim&\sum_{k=0,1}\n{w_{l} \pa_x^{k} \wt{f}_{0}}_{L_{x,v}^{\infty}}+
\sup_{s\in[0,t]}\n{g_{2}(s)}_{L_{x,v}^{2}}
 +\sup_{s\in[0,t]} e^{\la_{0} s}\n{\pa_{x}g_{2}(s)}_{L_{x,v}^{2}}+
 \sum_{k=0,1}\n{[g_{1},g_{2}]}_{\ol{X}_{k,t}}^{2}\notag\\
\lesssim&  \sum_{k=0,1}\n{w_{l} \pa_x^{k} \wt{f}_{0}}_{L_{x,v}^{\infty}}+
\al \sum_{k=0,1}\n{g_{2}}_{\ol{X}_{k,t}}+\ve\sum_{k=0,1}\n{[g_{1},g_{2}]}_{\ol{X}_{k,t}} +C_{\ve}\sum_{k=0,1}\n{[g_{1},g_{2}]}_{\ol{X}_{k,t}}^{4}\notag\\
 &+\sup_{s\in[0,t]} e^{\la_{0} s}\n{\pa_{x}g_{2}(s)}_{L_{x,v}^{2}}+\lrs{\sup_{s\in[0,t]}e^{\la_{0} s}\n{\pa_{x}g_{2}(s)}_{L_{x,v}^{2}}}^{2},\notag
\end{align}
where in the last inequality we used the interpolation inequality that $x^{2}\lesssim \ve x+C_{\ve}x^{4}$. For the higher energy estimate,
we use Lemma \ref{lemma:g2,higher energy} to get
\begin{align}\label{equ:main proof,higher,g2}
&\sup_{s\in[0,t]} e^{\la_{0} s}\n{\pa_{x}g_{2}(s)}_{L_{x,v}^{2}}\\
\lesssim&
\n{w_{l}\pa_{x}\wt{f}_{0}}_{L_{x,v}^{\infty}}+\al \n{g_{2}}_{\ol{X}_{1,t}}+
 \lrs{\sum_{k=0,1}\n{[g_{1},g_{2}]}_{\ol{X}_{k,t}}}^{2}\notag\\
\leq& \n{w_{l}\pa_{x}\wt{f}_{0}}_{L_{x,v}^{\infty}}+\al \n{g_{2}}_{\ol{X}_{1,t}}+
\ve\sum_{k=0,1}\n{[g_{1},g_{2}]}_{\ol{X}_{k,t}}+C_{\ve}\lrs{\sum_{k=0,1}\n{[g_{1},g_{2}]}_{\ol{X}_{k,t}}}^{4}.\notag
\end{align}
Noting that $\n{w_{l}\pa_{x}\wt{f}_{0}}_{L_{x,v}^{\infty}}^{2}\lesssim \n{w_{l}\pa_{x}\wt{f}_{0}}_{L_{x,v}^{\infty}}\lesssim 1$, we further obtain
\begin{align}\label{equ:main proof,higher,g2,square}
\lrs{\sup_{s\in[0,t]}e^{\la_{0} s}\n{\pa_{x}g_{2}(s)}_{L_{x,v}^{2}}}^{2}
\lesssim&
\n{w_{l}\pa_{x}\wt{f}_{0}}_{L_{x,v}^{\infty}}^{2}+\al^{2} \n{g_{2}}_{\ol{X}_{1,t}}^{2}+
 \lrs{\sum_{k=0,1}\n{[g_{1},g_{2}]}_{\ol{X}_{k,t}}}^{4}\\
\lesssim& \n{w_{l}\pa_{x}\wt{f}_{0}}_{L_{x,v}^{\infty}}+\al \n{g_{2}}_{\ol{X}_{1,t}}+
\lrs{\sum_{k=0,1}\n{[g_{1},g_{2}]}_{\ol{X}_{k,t}}}^{4}.\notag
\end{align}

Putting together estimates \eqref{equ:main proof,X,L2}, \eqref{equ:main proof,higher,g2}, and \eqref{equ:main proof,higher,g2,square}, we arrive at
\begin{align*}
\sum_{k=0,1}\n{[g_{1},g_{2}]}_{\ol{X}_{k,t}}
\leq& C\sum_{k=0,1}\n{w_{l} \pa_x^{k} \wt{f}_{0}}_{L_{x,v}^{\infty}}+
C\lrs{\sum_{k=0,1}\n{[g_{1},g_{2}]}_{\ol{X}_{k,t}}}^{4}.
\end{align*}
Therefore, provided that $\sum_{k=0,1}\n{w_{l} \pa_x^{k} \wt{f}_{0}}_{L_{x,v}^{\infty}}\ll 1$, we get
\begin{align*}
\sum_{k=0,1}\n{[g_{1},g_{2}]}_{\ol{X}_{k,t}}\leq 2C\sum_{k=0,1}\n{w_{l} \pa_x^{k} \wt{f}_{0}}_{L_{x,v}^{\infty}}.
\end{align*}
A standard continuity argument, together with the local well-poseness Lemma \ref{lemma:lwp,g12}, then allows us to extend the solution globally in time. Consequently, we derive the global stability estimate
\begin{align}\label{equ:global stability,g12}
\sup_{t\in[0,\infty)}\sum_{k=0,1}\n{[g_{1},g_{2}]}_{\ol{X}_{k,t}}\leq 2C\sum_{k=0,1}\n{w_{l} \pa_x^{k} \wt{f}_{0}}_{L_{x,v}^{\infty}}.
\end{align}
Since $e^{3\beta t}F(t,x,e^{\beta t}v)-G(v)=\wt{f}=g_{1}+\mu^{\frac{1}{2}}g_{2}$, estimates \eqref{equ:global stability,estimate,thm,P0}--\eqref{equ:global stability,estimate,thm,P1} follow from the global stability estimate \eqref{equ:global stability,g12}.

Next, we get into the analysis of the long-time behavior of the energy moments.
We recall that
\begin{align*}
U=\begin{pmatrix}
P_{0}^{x}E\\
P_{0}^{x}d_{12}\\
P_{0}^{x}d_{22}
\end{pmatrix},\quad
R=\begin{pmatrix}
0\\
\lra{P_{0}^{x}Q(\wt{f},\wt{f}),v_{1}v_{2}}\\
\lra{P_{0}^{x}Q(\wt{f},\wt{f}),v_{2}v_{2}}
\end{pmatrix},
\end{align*}
and
\begin{align*}
A_{\al}=2\be I+
\begin{pmatrix}
0&2\al &0\\
0&2b_{0} & \al \\
-\frac{2b_{0}}{3}&0&2b_{0}
\end{pmatrix},\quad
Q_{\al}^{-1}A_{\al}Q_{\al}=\operatorname{Diag}(A_{\al})=
\begin{pmatrix}
0&&\\
&\la_{2}&\\
&&\la_{3}
\end{pmatrix},
\end{align*}
where
\begin{align*}
Q_{\al}=
\begin{pmatrix}
6b_{0}& 0& 0\\
0& 1 &1\\
2b_{0}& -i\frac{\sqrt{6}}{6}& i\frac{\sqrt{6}}{6}
\end{pmatrix}+O(\al),\quad
Q_{\al}^{-1}=
\begin{pmatrix}
\frac{1}{6b_{0}}& 0& 0\\
-i\frac{\sqrt{6}}{6}& \frac{1}{2} &i\frac{\sqrt{6}}{2}\\
i\frac{\sqrt{6}}{6}& \frac{1}{2}& -i\frac{\sqrt{6}}{2}
\end{pmatrix}+O(\al).
\end{align*}
For brevity, we set
$S_{\al}(t)=e^{-tA_{\al}}$.
We observe the limit as $t\to \infty$
\begin{align}\label{equ:limit,St}
\lim_{t\to \infty}S_{\al}(t)=&\lim_{t\to \infty}Q_{\al}
\begin{pmatrix}
1&0&0\\
0&e^{-\la_{2}t}&0\\
0&0&e^{-\la_{3}t}
\end{pmatrix}Q_{\al}^{-1}
=Q_{\al}
\begin{pmatrix}
1&0&0\\
0&0&0\\
0&0&0
\end{pmatrix}
Q_{\al}^{-1}\\
=&\begin{pmatrix}
1& 0&0\\
0&0&0\\
\frac{1}{3}&0 &0
\end{pmatrix}+O(\al):=S_{\al}(\infty).\notag
\end{align}
By the Duhamel's formula, we have
\begin{align*}
U(t)=S_{\al}(t)U(0)+\int_{0}^{t}S_{\al}(t-s)R(s)ds.
\end{align*}
By \eqref{equ:estimate,duhamel,R123}, \eqref{equ:estimate,Ri}, \eqref{equ:g2,higher energy} and the global stability estimate \eqref{equ:global stability,g12}, we obtain the exponential decay estimate
\begin{align*}
\sup_{s\in[0,\infty)}e^{2\la_{0}s}\n{R(s)}_{L_{x}^{2}}\lesssim \lrs{\sum_{k=0,1}\n{w_{l} \pa_x^{k} \wt{f}_{0}}_{L_{x,v}^{\infty}}}^{2}\lesssim 1.
\end{align*}
Therefore, we can define the long-time limit state
\begin{align}\label{equ:long-time limit state}
U(\infty):=S_{\al}(\infty)U(0)+\int_{0}^{\infty}S_{\al}(\infty)R(s)ds.
\end{align}

Taking the difference between $U(t)$ and $U(\infty)$ yields
\begin{align*}
U(t)-U(\infty)=I_{1}+I_{2}+I_{3},
\end{align*}
where
\begin{align*}
I_{1}=&(S_{\al}(t)-S_{\al}(\infty))U(0),\\
I_{2}=&\int_{0}^{t}\lrc{S_{\al}(t-s)-S_{\al}(\infty)}R(s)ds,\\
I_{3}=&\int_{t}^{\infty}S_{\al}(\infty)R(s)ds.
\end{align*}
From \eqref{equ:limit,St}, we have
\begin{align*}
\n{S_{\al}(t)-S_{\al}(\infty)}_{op}\lesssim e^{-(\operatorname{Re}\la_{2}) t}+e^{-(\operatorname{Re}\la_{3}) t}\lesssim e^{-\la_{0}t},
\end{align*}
where we have used that $\la_{0}\leq \operatorname{Re}\la_{i}=2b_{0}$ for $i=2,3$. Then we estimate each term $I_{i}$ as follows
\begin{align*}
|I_{1}|\lesssim& \n{S_{\al}(t)-S_{\al}(\infty)}_{op}\abs{U(0)}\lesssim e^{-\la_{0}t},\\
|I_{2}|\lesssim& \int_{0}^{t}\n{S_{\al}(t-s)-S_{\al}(\infty)}_{op}|R(s)|ds
\lesssim\int_{0}^{t}e^{-\la_{0}(t-s)}e^{-2\la_{0}s}ds \leq e^{-\la_{0}t},\\
|I_{3}|\lesssim& \int_{t}^{\infty}|R(s)|ds\lesssim \int_{t}^{\infty}e^{-2\la_{0}s}ds\leq e^{-\la_{0}t}.
\end{align*}
Thus, we conclude that
\begin{align}\label{equ:exconvergence,energy moment}
\abs{U(t)-U(\infty)}\lesssim e^{-\la_{0}t}.
\end{align}
By the definitions \eqref{equ:def,energy moment} and \eqref{equ:hydrodynamic moments}, we have $\mathbf{E}_{\al}(t)=P_{0}^{x}E(t)$, which is the first component of $U(t)$. Therefore, by \eqref{equ:exconvergence,energy moment}, we immediately obtain $\mathbf{E}_{\al}(\infty)=P_{0}^{x}E(\infty)$ and
\eqref{equ:exconvergence,long-time,state}. Multiplying both sides of \eqref{equ:exconvergence,long-time,state} by $e^{2\beta t}$, with $\beta\ll \la_{0}$, we arrive at \eqref{equ:long-time,state,energy,growth}.

Finally, we address the non-negativity of the global solution $F(t)$ via a standard iterative method (see
for example
\cite[Section 4]{Guo03}). We consider the iterative scheme defined by
\begin{equation*}
\left\{
\begin{aligned}
\pa_{t}F^{n+1} + v\cdot \nabla_{x}F^{n+1} - \alpha v_{2}\pa_{v_{1}}F^{n+1}+Q^{-}(F^{n+1},F^{n}) =& Q^{+}(F^{n}, F^{n}),\\
F^{n+1}(0,x,v)=&F_{0}(x,v)\geq 0,
\end{aligned}
\right.
\end{equation*}
which yields that $F^{n+1}(t,x,v)\geq 0$. Following an argument analogous to the local well-posedness theory established in Section \ref{sec:Local Well-posedness}, we deduce that the solution $F(t)$ is non-negative on a local time interval.
Since the non-negativity is a local property, we can then bootstrap to fill the entire time interval $[0,\infty)$.

Hence, we have completed the proof of Theorem \ref{thm:main theorem}.
\end{proof}

\appendix
\section{Estimates for Collision Operators}\label{section:Estimates for Collision Operators}
In this section, we collect useful and well-established estimates concerning the collision operators, which are employed in our analysis.
\begin{lemma}\label{Lemma,L,K,estimates}
Let the linearised operators $L$ and $K$ be given by \eqref{equ:notation,L,K}--\eqref{equ:notation,K}.
In the Maxwell molecular case, there is a constant $\delta_0>0$ such that
\begin{align}\label{equ,L,lower bound}
\langle L f, f\rangle=\left\langle L \mathbf{P}_{1} f, \mathbf{P}_{1} f\right\rangle \geq \delta_0\left\|\mathbf{P}_{1} f\right\|_{L_{v}^{2}}^2.
\end{align}
The operator $K$ can be expressed as an integral operator of the form
\begin{align*}
K f(v)=\int_{\mathbb{R}^3} k\left(v, v_*\right) f\left(v_*\right) d v_*
\end{align*}
with the kernel function satisfying
\begin{align*}
\left|k\left(v, v_*\right)\right| \leq C\lrs{1+\left|v-v_*\right|^{-2}} e^{-\frac{1}{8}\left|v-v_*\right|^2-\frac{1}{8} \frac{\left| \left| v\right|^2-\left|v_*\right|^2\right|^2}{\left|v-v_*\right|^2}} .
\end{align*}
Moreover, for $l\geq0$, it holds that
\begin{align}\label{equ:kernel estimate,K,pointwise}
w_{l}(v) |k\left(v, v_*\right)| w_{-l}\left(v_*\right)\leq C_{l}(1+|v-v_{*}|^{-2})e^{-\frac{|v-v_{*}|^{2}}{16}}
\end{align}
and
\begin{align}\label{equ:kernel estimate,K,L1}
\int_{\mathbb{R}^3} w_{l}(v) |k\left(v, v_*\right)| w_{-l}\left(v_*\right) e^{\frac{\delta\left|v-v_*\right|^2}{8}} d v_* \leq \frac{C_{\delta,l}}{1+|v|}
\end{align}
for $\delta \geq0$ small enough.
\end{lemma}
\begin{proof}
These collision estimates are well-known.
See for example \cite{CIP94,Guo10}.
\end{proof}

\begin{lemma}\label{lemma:Gamma,bilinear estimate}
Let the collision operator $\Gamma(f,g)$ be given by \eqref{equ:notation,Gamma}.
 In the Maxwell molecular case, for $l \geq 0$, we have
\begin{align}
\left\|w_{l}  \Gamma(f, g)\right\|_{L_v^2} \leq& C \left\|w_{l}  f\right\|_{L_v^2}\left\|w_{l} g\right\|_{L_v^2},
\label{equ:Gamma,bilinear estimate,L2}\\
\left\|w_{l}  \Gamma(f, g)\right\|_{L_{v}^{\infty}} \leq& C \left\|w_{l}  f\right\|_{L_{v}^{\infty}}\left\|w_{l} g\right\|_{L_{v}^{\infty}}.
\label{equ:Gamma,bilinear estimate,Linf}
\end{align}
 In particular, we have
\begin{align}
\n{w_{l}Kf}_{L_{v}^{\infty}}\leq C\n{w_{l}f}_{L_{v}^{\infty}},
\end{align}
where $Kf=\nu_{0}f-Lf$ is defined by \eqref{equ:notation,L,K}.
\end{lemma}
\begin{proof}
The proof of \eqref{equ:Gamma,bilinear estimate,L2} and \eqref{equ:Gamma,bilinear estimate,Linf} is similar as that of
\cite[Lemma 2.3]{Guo02}
and \cite[Lemma 5]{Guo10}, respectively. We therefore omit the details for brevity.
\end{proof}

\begin{lemma}\label{lemma:Q,bilinear estimate,Lp}
 In the Maxwell molecular case, it holds that
 \begin{align}
\n{Q^{-}(f,g)}_{L_{v}^{p}}\lesssim& \n{f}_{L_{v}^{1}}\n{g}_{L_{v}^{p}},\label{equ:loss term,estimate,f,g} \quad \text{for $p\in[1,\infty]$},\\
\n{Q^{+}(f,g)}_{L_{v}^{p}}\lesssim& \n{f}_{L_{v}^{1}}\n{g}_{L_{v}^{p}},\label{equ:gain term,estimate,f,g}\quad \text{for $p\in[1,\infty]$}.
\end{align}
In particular, we have
\begin{align}
\n{K_{\mu}f}_{L_{v}^{2}}\lesssim& \n{f}_{L_{v}^{1}}+\n{f}_{L_{v}^{2}},\label{equ:Kmu,estimate,Lp}
\end{align}
where $K_{\mu}f=Q(f,\mu)+Q^{+}(\mu,f)$ is defined by \eqref{equ:notation,Kmu}.
Moreover, for $l>\frac{3}{2}$, we have
\begin{align}\label{equ:closed estimate,weighted}
\left\|w_{l}  Q^{\pm}\left(f, g\right)\right\|_{L_{v}^{\infty}} \lesssim& \left\|w_{l} f\right\|_{L_{v}^{\infty}}\left\|w_{l} g\right\|_{L_{v}^{\infty}},\\
 \n{w_{l}K_{\mu}f}_{L_{v}^{\infty}}\lesssim &\n{w_{l}f}_{L_{v}^{\infty}}.\label{equ:Kmu,estimate,Linf}
\end{align}
\end{lemma}
\begin{proof}
The loss term estimate \eqref{equ:loss term,estimate,f,g} is straightforward. The gain term estimate \eqref{equ:gain term,estimate,f,g} follows from \cite[Theorem 1]{ACG10} or \cite[Theorem 2.1]{MV04}. The estimates \eqref{equ:loss term,estimate,f,g}--\eqref{equ:gain term,estimate,f,g} then yield estimates \eqref{equ:Kmu,estimate,Lp}--\eqref{equ:closed estimate,weighted}.
\end{proof}

\begin{lemma}[{\hspace{-0.01em}\cite[Proposition 2.1]{DL21}}]\label{lemma:large velocity,K}
For any positive integer $k \geq0$, there is $C>0$ such that for any arbitrarily large $l>0$, there is $M=M(l)>0$ such that it holds that
$$
\sup _{|v| \geq M} w_{l}\left|\nabla_v^k K_{\mu} f\right| \leq \frac{C}{l} \sum_{0 \leq k^{\prime} \leq k}\left\|w_{l} \nabla_v^{k^{\prime}} f\right\|_{L_{v}^{\infty}}.
$$
\end{lemma}

\medskip
\noindent {\bf Acknowledgment:}\,
The research of RJD was partially supported by the General Research Fund (Project No.~14303523) from RGC of Hong Kong and also by the grant from the National Natural Science Foundation of China (Project No.~12425109). The research of SQL was partially supported by the grant from the National Natural Science Foundation of China (Project No.~12325107). The research of SLS was partially supported by the National Key R\&D Program of China (Project No.~2024YFA1015500).

\medskip
\noindent\textbf{Data Availability Statement:}
Data sharing is not applicable to this article as no datasets were generated or analysed during the current study.

\noindent\textbf{Conflict of Interest:}
The authors declare that they have no conflict of interest.

%\bibliographystyle{abbrv}
%%\bibliographystyle{plain}
%%\nocite{*}
%\bibliography{references}

\end{document}